\documentclass[9pt,shortpaper,twoside,web]{ieeecolor}
\usepackage{generic}
\usepackage{cite}

\usepackage{amsmath,amssymb,amsfonts}
\usepackage{algorithmic}
\usepackage{graphicx}
\usepackage{textcomp}
\usepackage{booktabs}
\usepackage{bm}
\usepackage{color}
\usepackage{setspace}
\usepackage{caption}
\usepackage{subfigure}
\usepackage{multirow}
\usepackage{diagbox}
\usepackage{subcaption}
\usepackage{algorithm, algorithmic}

\newtheorem{lemma}{Lemma}
\newtheorem{theorem}{Theorem}

\newtheorem{remark}{Remark}
\newtheorem{assumption}{Assumption}

\newtheorem{proposition}{Proposition}

\def\BibTeX{{\rm B\kern-.05em{\sc i\kern-.025em b}\kern-.08em
    T\kern-.1667em\lower.7ex\hbox{E}\kern-.125emX}}
\allowdisplaybreaks
\begin{document}
\title{Asynchronous Push-sum Dual Gradient Algorithm in Distributed Model Predictive Control}
\author{Pengbiao Wang, Xuemei Ren, and Dongdong Zheng
\thanks{This work was supported by the National Natural Science Foundation of China under Grant 62273050, and Grant 61973036. (Corresponding author: Xuemei Ren.)}
\thanks{The authors are with School of Automation, Beijing Institute of Technology, Beijing, 100081, China (e-mail: pwang87@cityu.edu.hk; xmren@bit.edu.cn; ddzheng@bit.edu.cn).}}
\maketitle

\begin{abstract}
This paper studies the distributed model predictive control (DMPC) problem for distributed discrete-time linear systems with both local and global constraints over directed communication networks. We establish an optimization problem to formulate the DMPC policy, including the design of terminal ingredients. To cope with the global constraint, we transform the primal optimization problem into its dual problem. Then, we propose a novel asynchronous push-sum dual gradient (APDG) algorithm with an adaptive step-size scheme to solve this dual problem in a fully asynchronous distributed manner. The proposed algorithm does not require synchronous waiting and any form of coordination, which greatly improves solving efficiency. We prove that the APDG algorithm converges at an $\mathit{R}$-linear rate as long as the step-size does not exceed the designed upper bound. Furthermore, we develop a distributed termination criterion to terminate the APDG algorithm when its output solution satisfies the specified suboptimality and the global constraint, thereby avoiding an infinite number of iterations. The recursive feasibility and the stability of the closed-loop system are also established. Finally, a numerical example is provided to clarify and validate our theoretical findings.
\end{abstract}

\begin{IEEEkeywords}
Distributed model predictive control, asynchronous push-sum dual gradient algorithm, $\mathit{R}$-linear convergence rate, adaptive step-size scheme, distributed termination criterion.
\end{IEEEkeywords}
\vspace{-0.4cm}
\section{Introduction}
\IEEEPARstart{M}{odel} predictive control (MPC) is a highly successful control methodology due to its ability to simultaneously optimize performance, handle constraints, and ensure system stability \cite{Qin2003A}. Over the past decade, due to the growing demand for controlling geographically isolated and/or large-scale distributed systems, DMPC has emerged and attracted considerable attention. Its research hotspots include handling local constraints and implementing cooperative/noncooperative control (see \cite{Scattolini2009Architectures,Conte2016distributed,Farina2012Distributed} and references therein).

Another active topic in DMPC is dealing with global constraints. The main challenge is to satisfy these constraints in a distributed manner. To tackle this challenge, the authors of \cite{Richards2007Robust} and \cite{Trodden2013Cooperative} proposed sequential DMPC methods that decompose the overall  problem into several smaller subproblems, where each subsystem solves its subproblem in a predefined sequential order. As an extension of sequential DMPC methods, parallelized DMPC methods \cite{Trodden2014Feasible,Gro2016} allow all subsystems to solve their respective subproblems simultaneously, thereby improving computational efficiency. However, parallelized DMPC methods require each subsystem to communicate directly with all other subsystems, which imposes stringent requirements on communication networks, even for medium-scale systems. Furthermore, the optimality of the overall systems may not be guaranteed, because the aforementioned methods \cite{Richards2007Robust,Trodden2013Cooperative,Trodden2014Feasible,Gro2016} do not explicitly pursue it.

In recent years, iterative DMPC approaches based on the alternating direction method of multipliers or dual decomposition have been developed and reported in \cite{Wang2017,Li2021Distributed ,Doan2011}. By relying on the dual problem of the DMPC problem, these approaches \cite{Wang2017,Li2021Distributed,Doan2011} address the optimality of the overall systems and guarantee the satisfaction of the global constraints; however, their convergence rates are known to be $\mathcal{O}(1/k)$. To improve the convergence rate, Giselsson et al. \cite{Giselsson2013,Giselsson2013A} have employed an accelerated gradient technique in dual decomposition methods to accelerate the convergence of dual variables, while Wang and Ong \cite{Wang2018Ong} have proposed a distributed fast dual gradient algorithm that has been successfully applied in iterative DMPC approaches. However, these algorithms \cite{Giselsson2013,Giselsson2013A,Wang2018Ong} converge at a rate of $\mathcal{O}(1/k^{2})$, which is still a sublinear convergence rate. Moreover, it should be noted that the aforementioned algorithms \cite{Wang2017,Li2021Distributed,Doan2011,Giselsson2013,Giselsson2013A,Wang2018Ong} are only applicable to the DMPC problem with global constraints over undirected communication networks. More recently, other distributed algorithms, such as the push-sum dual gradient algorithm \cite{Jin2021} and the privacy-preserving distributed projected gradient algorithm \cite{Zhao2022}, have been developed for the DMPC problem with global constraints over directed communication networks. Typically, these algorithms \cite{Jin2021,Zhao2022} only converge at a sublinear rate because they use diminishing step-sizes and local gradient information. It  should be noted that all the aforementioned algorithms \cite{Wang2017,Li2021Distributed,Doan2011,Giselsson2013,Giselsson2013A,Wang2018Ong,Jin2021,Zhao2022} achieve only sublinear convergence rates for the DMPC problems with both local and global constraints over undirected/directed graphs. It is well-known that algorithms with linear convergence rates are highly preferable for these problems, as they can greatly reduce the number of iterations and shorten the time in solving optimization problems.

On the other hand, the aforementioned algorithms \cite{Wang2017,Li2021Distributed,Doan2011,Giselsson2013,Giselsson2013A,Wang2018Ong,Jin2021,Zhao2022} are synchronous, as they require a central unit for coordination and clock synchronization. It is noted that these synchronous algorithms \cite{Wang2017,Li2021Distributed,Doan2011,Giselsson2013,Giselsson2013A,Wang2018Ong,Jin2021,Zhao2022}
result in significant idle time for subsystems with rapid update rates, thereby potentially reducing computational efficiency. While the asynchronous algorithm \cite{Zhang2019You} can address this issue, it only converges at a sublinear rate of $\mathcal{O}(\ln(k)/\sqrt{k})$ and is designed for unconstrained optimization problems. To the best of our knowledge, there are no asynchronous distributed algorithms with linear convergence rates
for the DMPC problem with both local and global constraints over directed communication networks; consequently, their development remains a challenging and open research area.

Motivated by these discussions, this paper proposes an APDG algorithm with a linear convergence rate for the DMPC problem with both local and global constraints over directed communication networks. The main contributions are as follows:

(1) A novel APDG algorithm with an adaptive step-size scheme is developed to solve the dual problem of the DMPC problem in a fully asynchronous distributed manner. We prove
that the proposed APDG algorithm converges exactly to the optimal solution at an $R$-linear convergence rate of $\mathcal{O}(\delta^{k})$, provided that the fixed step-size satisfies the specifically designed condition. Compared with existing algorithms \cite{Wang2017,Li2021Distributed,Doan2011,Giselsson2013,Giselsson2013A,Wang2018Ong,Zhang2019You,Nedic2014}, our APDG algorithm is suitable for the DMPC  problem under both local and global constraints over directed communication networks, which is much more challenging. In contrast to existing algorithms \cite{Jin2021,Zhang2019You,Nedic2014} with sublinear convergence rates, our APDG algorithm not only has the advantages of asynchronous updates but, more importantly, converges to the optimal solution at an $R$-linear rate. This implies that it requires fewer iterations and a significantly reduced optimization time.

(2) To further reduce computational requirements, we develop a distributed termination criterion to stop iterations of the APDG algorithm when the global constraint is satisfied and the output solution reaches a specified suboptimality level. Compared with existing termination criteria \cite{Wang2017,Li2021Distributed,Wang2018Ong,Jin2021}, the proposed criterion does not require any form of coordination and has fewer design parameters,
thereby significantly simplifying implementation.

The rest of this paper is organized as follows. Sections \ref{S2} and \ref{S3} introduce the problem formulation and a novel APDG algorithm, respectively. Section \ref{S4} presents the convergence analysis and a distributed termination criterion for the APDG algorithm. Our APDG-based iterative DMPC approach, its recursive feasibility and the stability of the resulting closed-loop system are given in Section \ref{S5}. Section \ref{S6} provides a numerical study. Finally, Section \ref{S7} summarizes our work.

\textbf{Notation:} Let $\mathbb{N}$ be the set of nonnegative integers. For a square matrix $P$, $P>0$ implies that $P$ is a positive definite matrix. For two nonnegative integers $M<L$, we define $\mathbb{Z}^{L}_{M}\triangleq\{M,M+1,\cdots,L\}$. For two vectors $x$ and $y$, $x\preceq y$ (or $x\prec y$) implies that all the entries of $x$ are less than or equal to (or less than) all the entries of $y$. For a set $\mathcal{S}$, we write $|\mathcal{S}|$ for its cardinality. For a matrix $P$, $p_{ij}$ and $\|P\|_{2}$ represent its $(i,j)$-th entry and induced 2-norm, respectively. The superscript `$\top$' denotes the transpose of a matrix or vector. For a real symmetric matrix $A$, let $\lambda_{\min}(A)$ and $\lambda_{\max}(A)$ be its minimum and maximum eigenvalues, respectively. We write $\text{col}(x^{1},\cdots,x^{n})=[(x^{1})^{\top},\cdots,(x^{n})^{\top}]^{\top}$.
For a vector $x$, the symbols $\|x\|_{2}$, $\|x\|_{1}$ and $\|x\|_{P}\triangleq\sqrt{x^{\top}Px}$ with $P>0$ denote its Euclidean-, $\ell_{1}$- and $P$-norm, respectively. The symbol $\otimes$ refers to the Kronecker product. The projection of a vector
$\nu$ on the set $\mathcal{X}$ is denoted by $\text{proj}_{\mathcal{X}}(\nu)=\arg\min\limits_{x\in\mathcal{X}}\|x-\nu\|_{2}$. $\mathbf{1}_{n}$ and $\mathbf{0}_{n}$ are the $n$-dimensional column vectors with all entries being 1 and 0, respectively. The subscript will be omitted when the dimension is clear from the context.

\section{Problem Formulation}\label{S2}
\subsection{System and Constraints}
Consider a distributed system consisting of $M$ discrete-time linear subsystems, each of which is described by
\begin{align}\label{system model}
x^{i}(t+1)=A^{i}x^{i}(t)+B^{i}u^{i}(t), \ i\in \mathbb{Z}^{M}_{1},
\end{align}
where $x^{i}(t)\in\mathbb{R}^{n_{i}}$ and $u^{i}(t)\in\mathbb{R}^{m_{i}}$ represent the state and input at time $t$, respectively. $A^{i}\in\mathbb{R}^{n_{i}\times n_{i}}$ and $B^{i}\in\mathbb{R}^{n_{i}\times m_{i}}$ are the state and input matrices, respectively. Subsystem $i$ is subject to the following local constraints:
\begin{align}\label{local constraint}
x^{i}(t)\in\mathcal{X}^{i},\ u^{i}(t)\in\mathcal{U}^{i},
\end{align}
where both $\mathcal{X}^{i}$ and $\mathcal{U}^{i}$ are polytopes that include the origin in their interior. Moreover, we impose the following global constraint on all subsystems:
\begin{align}\label{global constraint}
\sum\limits_{i=1}^{M}\big(\mathcal{C}^{i}x^{i}(t)+\mathcal{D}^{i}u^{i}(t)\big)\preceq \mathbf{1}_{\rho},
\end{align}
where $\mathcal{C}^{i}\!\in\!\mathbb{R}^{\rho\times n_{i}}$ and $\mathcal{D}^{i}\!\in\! \mathbb{R}^{\rho\times m_{i}}$ are known only to subsystem $i$.

The control objective is to develop an iterative DMPC approach that stabilizes the state of each subsystem at the equilibrium point. To this end, the following standard assumption is made.

\begin{assumption}\!\!\label{assumption 1}\cite{Wang2017,Li2021Distributed}
For any subsystem $i$, the state $x^{i}(t)$ is measurable and the matrix pair $(A^{i},B^{i})$ is stabilizable.
\end{assumption}

\subsection{Network Topology: A Digraph}
We use the digraph $\mathcal{G}=(\mathcal{V},\mathcal{E})$ to represent the network topology among $M$ subsystems, in which $\mathcal{V}\triangleq\{1,2,\cdots,M\}$ is the vertex (or subsystem) set, and $\mathcal{E}\subseteq\mathcal{V}\times\mathcal{V}$ is the edge set. Specifically, the edge set denotes unidirectional communication links, that is, subsystem $i$ can send information to subsystem $j$ if $(i\rightarrow j)\in\mathcal{E}$. For subsystem $i$, we define its sets of in-neighbors and out-neighbors as $N^{i}_{\text{in}}\triangleq\{j|(j\rightarrow i)\in\mathcal{E}\}\cup\{i\}$ and $N^{i}_{\text{out}}\triangleq\{j|(i\rightarrow j)\in\mathcal{E}\}\cup\{i\}$, respectively. $|N^{i}_{\text{out}}|$ is the out-degree of subsystem $i$. The weight matrix of the digraph $\mathcal{G}$ is denoted by $\mathcal{A}=[a_{ij}]_{i,j=1}^{M}$, where $a_{ij}=1/|N^{j}_{\text{out}}|$ if $j\in N^{i}_{\text{in}}$, and $a_{ij}=0$ otherwise.
It is noted that the weight matrix $\mathcal{A}$ is column-stochastic, i.e., $\mathbf{1}^{\top}\mathcal{A}=\mathbf{1}^{\top}$ (or equivalently $\sum_{i=1}^{M}a_{ij}=1,\ \forall j\in\mathcal{V}$) \cite{Nedic2014}.

\begin{assumption}\label{assumption 2} (1) The digraph $\mathcal{G}$ is fixed and strongly connected.
  (2) There exists a sufficiently small constant $\bar{a}>0$ such that $a_{ij}\geq\bar{a}$ for any $j\in N^{i}_{\text{in}}$.
\end{assumption}
\begin{assumption}\label{assumption 3}
The communication delay is time-varying but uniformly upper-bounded, i.e., for any $(j\rightarrow i)\in\mathcal{E}$, the communication delay $\tau_{ij}$ satisfies $0\leq\tau_{ij}\leq\tau$, where $\tau>0$.
\end{assumption}

Assumption \ref{assumption 2} is common in distributed optimization over digraphs \cite{Zhang2019You}. Moreover, Assumption \ref{assumption 3} is quite reasonable \cite{Zhang2019You} and serves as a prerequisite for Lemma \ref{lemma 1} in Section \ref{S3.2}.

\subsection{Optimization Problem}
In this subsection, we formulate the DMPC optimization problem. For subsystem $i$, define its cost function as
\begin{align}\label{cost}
 J^{i}(x^{i},\boldsymbol{u}^{i})\triangleq&\sum\limits_{\ell=0}^{N-1}\left(\|x^{i}_{\ell}\|_{Q^{i}}^{2}+\|u^{i}_{\ell}\|_{R^{i}}^{2}\right)+
\|x^{i}_{N}\|_{P^{i}}^{2},
\end{align}
where $N>0$ is the prediction horizon, $x^{i}$ is the state of subsystem $i$ at time $t$, $\boldsymbol{u}^{i}\triangleq\text{col}(u^{i}_{0},u^{i}_{1},\cdots,u^{i}_{N-1})$ is the collection of $N$ predictive inputs at time $t$, and $x^{i}_{\ell}$ is the predictive state evolving along the predictive dynamics $x^{i}_{\ell+1}=A^{i}x^{i}_{\ell}+B^{i}u^{i}_{\ell}$ with $x^{i}_{0}=x^{i}$. For given weight matrices $Q^{i}>0$ and $R^{i}>0$, we obtain the weight matrix $P^{i}>0$ by solving the discrete-time algebraic Riccati equation (ARE) $P^{i}\!=\!(A^{i})^{\top}P^{i}A^{i}\!-\!(A^{i})^{\top}P^{i}B^{i}(R^{i}+(B^{i})^{\top}P^{i}B^{i})^{\!-\!1}(B^{i})^{\top}P^{i}A^{i}\!+Q^{i}$ \cite{Rawlings2009}. At time $t$, we define the collection of $N+1$ predictive states as $\boldsymbol{x}^{i}\triangleq\text{col}(x_{0}^{i},x_{1}^{i},\cdots,x_{N}^{i})$. And it can be expressed as
\begin{align*}
\boldsymbol{x}^{i}=\bar{A}^{i}x^{i}+\bar{B}^{i}\boldsymbol{u}^{i},
\end{align*}
where \begin{align*}
        \bar{A}^{i}=\begin{bmatrix}
          I \\
          A^{i} \\
          (A^{i})^2 \\
          \vdots \\
          (A^{i})^{N}
        \end{bmatrix} \ \text{and} \ \bar{B}^{i}=\begin{bmatrix}
                                       0 & 0 & \cdots & 0 \\
                                       B^{i} & 0 & \cdots & 0 \\
                                       A^{i}B^{i} & B^{i} & \cdots & 0 \\
                                       \vdots & \vdots & \ddots & \vdots \\
                                       (A^{i})^{N-1}B^{i} & (A^{i})^{N-2}B^{i} & \cdots & B^{i}
                                     \end{bmatrix}.
      \end{align*}
Then, the cost function \eqref{cost} can be rewritten as
\begin{align*}
J^{i}(x^{i},\boldsymbol{u}^{i})=&(\boldsymbol{x}^{i})^{\top}\underbrace{\begin{bmatrix}
                                                             Q^{i} & &    &  \\
                                                              &   \ddots &  &  \\
                                                              &    & Q^{i} &  \\
                                                              &    &  & P^{i}
                                                           \end{bmatrix}}_{\mathcal{Q}^{i}}\boldsymbol{x}^{i}\\
                   &+(\boldsymbol{u}^{i})^{\top}\underbrace{\begin{bmatrix}
                                                              R^{i} &  &  &  \\
                                                               & R^{i} &  &  \\
                                                               &  & \ddots &  \\
                                                               &  &  & R^{i}
                                                            \end{bmatrix}}_{\mathcal{R}^{i}}\boldsymbol{u}^{i}\\
=&(\boldsymbol{x}^{i})^{\top}\mathcal{Q}^{i}\boldsymbol{x}^{i}+(\boldsymbol{u}^{i})^{\top}\mathcal{R}^{i}\boldsymbol{u}^{i} \\
=&(\bar{A}^{i}x^{i}+\bar{B}^{i}\boldsymbol{u}^{i})^{\top}\mathcal{Q}^{i}(\bar{A}^{i}x^{i}+\bar{B}^{i}\boldsymbol{u}^{i})+(\boldsymbol{u}^{i})^{\top}\mathcal{R}^{i}\boldsymbol{u}^{i}\\
=&(\boldsymbol{u}^{i})^{\top}\mathcal{W}^{i}\boldsymbol{u}^{i}+2(\boldsymbol{u}^{i})^{\top}\mathcal{F}^{i}x^{i}+(x^{i})^{\top}\mathcal{H}^{i}x^{i},
\end{align*}
where
\begin{align*}
\mathcal{Q}^{i}=\text{diag}(\overbrace{Q^{i},\cdots,Q^{i}}^{N \text{ blocks}}, P^{i}), \ \mathcal{R}^{i}=\text{diag}(\overbrace{R^{i},\cdots,R^{i}}^{N \text{ blocks}}),
\end{align*}
$\mathcal{W}^{i}=(\bar{B}^{i})^{\top}\mathcal{Q}^{i}\bar{B}^{i}+\mathcal{R}^{i}$, $\mathcal{F}^{i}=(\bar{B}^{i})^{\top}\mathcal{Q}^{i}\bar{A}^{i}$ and $\mathcal{H}^{i}=(\bar{A}^{i})^{\top}\mathcal{Q}^{i}\bar{A}^{i}$. It should be noted that the cost function $J^{i}(x^{i},\boldsymbol{u}^{i})$, being quadratic in $\boldsymbol{u}^{i}$, is differentiable, and its second derivative (Hessian matrix) is $2\mathcal{W}^{i}$. It follows from \cite{Horn1990} that the cost function $J^{i}(x^{i},\boldsymbol{u}^{i})$ is $\mu^{i}$-strongly convex with respect to $\boldsymbol{u}^{i}$ and has an $l^{i}$-Lipschitz continuous gradient, where $\mu^{i}=\lambda_{\min}(2\mathcal{W}^{i})$, $l^{i}=\lambda_{\max}(2\mathcal{W}^{i})$ and $\mu^{i}<l^{i}$.

To formulate the subsequent DMPC optimization problem, we define the maximal positively invariant set as follows:
\begin{align}
\mathcal{X}^{i}_{\text{MPI}}\triangleq\{x^{i}_{\ell}: x^{i}_{\ell}\in\mathcal{X}^{i},K^{i}x^{i}_{\ell}\in\mathcal{U}^{i},x^{i}_{\ell+1}=A^{i}_{K}x^{i}_{\ell},\forall \ell\in\mathbb{Z}_{0}^{N-1}\},
\end{align}
where $A^{i}_{K}\!=\! A^{i}\!+\!B^{i}K^{i}$, and $K^{i}=-(R^{i}+(B^{i})^{\top}P^{i}B^{i})^{-1}(B^{i})^{\top}$\\$P^{i}A^{i}$ is the optimal control gain obtained from the linear quadratic regulator. Then, we define the terminal set $\mathcal{X}^{i}_{f}$ satisfying the following properties:
\begin{align}\label{terminal set}
\mathcal{X}^{i}_{f}\subset\mathcal{X}^{i}_{\text{MPI}}, A^{i}_{K}x^{i}_{\ell}\in\mathcal{X}^{i}_{f},\mathcal{A}^{i}_{K}x^{i}_{\ell}\preceq\sigma\mathbf{1}_{\rho} \ \text{for} \ \forall x^{i}_{\ell}\in\mathcal{X}^{i}_{f},
\end{align}
where $\mathcal{A}^{i}_{K}=\mathcal{C}^{i}+\mathcal{D}^{i}K^{i}$, $\sigma=\frac{1}{M}-(N+1)\gamma$, and $\gamma$ is an appropriate constant satisfying $0<\gamma<\frac{1}{M(N+1)}$. By utilizing Algorithm 3.1 in \cite{Gilbert1991} and the Multi-Parametric Toolbox, the sets $\mathcal{X}^{i}_{\text{MPI}}$ and   $\mathcal{X}^{i}_{f}$ can be easily obtained. Note that if the predictive states of all subsystems enter their respective terminal sets, the global constraint can be naturally satisfied under $u^{i}_{\ell}=K^{i}x^{i}_{\ell}$, which can be equivalently written as
$\sum_{i=1}^{M}\left(\mathcal{C}^{i}+\mathcal{D}^{i}K^{i}\right)x^{i}_{\ell}\preceq \sigma M\mathbf{1}_{\rho}\prec\mathbf{1}_{\rho}$
for all $x^{i}_{\ell}\in\mathcal{X}^{i}_{f}$, $\ell\in\mathbb{Z}_{0}^{N-1}$. To proceed, we define the following local constraint set:
\begin{align}\label{compact set}
\nonumber\mathcal{U}^{i}_{T}(x^{i})\triangleq\Big\{\boldsymbol{u}^{i}: \ &x^{i}_{\ell}\in\mathcal{X}^{i},u^{i}_{\ell}\in\mathcal{U}^{i},x^{i}_{N}\in\mathcal{X}^{i}_{f},x^{i}_{\ell+1}=A^{i}x^{i}_{\ell}\\
&+B^{i}u^{i}_{\ell},x^{i}_{0}=x^{i},\forall \ell\in\mathbb{Z}_{0}^{N-1}\Big\}.
\end{align}

To ensure the satisfaction of the global constraint \eqref{global constraint} under the early termination of the subsequent algorithm, we formulate the following DMPC optimization problem by tightening the global constraint \eqref{global constraint} and combining \eqref{cost} and \eqref{compact set},
\begin{subequations}\label{optimization problem1}
\begin{align}
\mathcal{P}(x):&\min\limits_{\boldsymbol{u}^{i}}\sum\limits_{i=1}^{M}J^{i}(x^{i},\boldsymbol{u}^{i}),\\
\text{s.t.} \quad& \boldsymbol{u}^{i}\in\mathcal{U}^{i}_{T}(x^{i}), \forall  i\in\mathcal{V},\label{constraint 0}\\
\quad\quad&\sum\limits_{i=1}^{M}\mathfrak{g}^{i}(x^{i},\boldsymbol{u}^{i})\preceq b(\epsilon),\label{tightening constraint}
\end{align}
\end{subequations}
where $x=\text{col}(x^{1},x^{2},\cdots,x^{M})$,  $\mathfrak{g}^{i}(x^{i},\boldsymbol{u}^{i})\triangleq\text{col}(\mathcal{C}^{i}x^{i}_{0}+\mathcal{D}^{i}u^{i}_{0},\cdots,
\mathcal{C}^{i}x^{i}_{N-1}+\mathcal{D}^{i}u^{i}_{N-1})$,  $b(\epsilon)\triangleq\text{col}((1-M\epsilon)\mathbf{1}_{\rho},\cdots,(1-MN\epsilon)\mathbf{1}_{\rho})$, and $0<\epsilon\leq\gamma$. Here, $\mathfrak{g}^{i}(x^{i},\boldsymbol{u}^{i})$ can also be written in a more compact form as follows

\begin{align*}
\mathfrak{g}^{i}(x^{i},\boldsymbol{u}^{i})\triangleq F^{i}x^{i}+H^{i}\boldsymbol{u}^{i},
\end{align*}
where $F^{i}\in\mathbb{R}^{N\rho\times n_{i}}$ and $H^{i}\in\mathbb{R}^{N\rho\times Nm_{i}}$ are given by
\begin{align*}
F^{i}&=\begin{bmatrix}
 \mathcal{C}^{i} \\
  \mathcal{C}^{i}A^{i} \\
  \mathcal{C}^{i}(A^{i})^{2} \\
  \vdots \\
  \mathcal{C}^{i}(A^{i})^{N-1}
\end{bmatrix},\\
H^{i}=&\begin{bmatrix}
         \mathcal{D}^{i} & 0 & 0 & \cdots & 0 \\
         \mathcal{C}^{i}B^{i} & \mathcal{D}^{i} & 0 & \cdots & 0 \\
         \mathcal{C}^{i}A^{i}B^{i} & \mathcal{C}^{i}B^{i} & \mathcal{D}^{i} & \cdots & 0 \\
         \vdots & \vdots & \vdots & \ddots & \vdots \\
         \mathcal{C}^{i}(A^{i})^{N-2}B^{i} & \mathcal{C}^{i}(A^{i})^{N-3}B^{i} & \mathcal{C}^{i}(A^{i})^{N-4}B^{i} & \cdots & \mathcal{D}^{i}
       \end{bmatrix}.
\end{align*}

Now we define a feasible region for the problem $\mathcal{P}(x)$
\vspace{-0.1cm}
\begin{align}
\mathcal{D}(x)\triangleq\{x: \mathcal{P}(x) \  \text{is feasible}\}.
\end{align}
We assume that the initial state is in the feasible region.

Since the global constraint \eqref{tightening constraint} exists, the existing Optimization Toolbox cannot be directly used to solve the problem $\mathcal{P}(x)$ in a distributed manner. Inspired by the Lagrangian method, in the next subsection, we will transform the problem $\mathcal{P}(x)$ into its dual problem.

\subsection{Dual Problem}\label{S2D}
The dual function of the problem $\mathcal{P}(x)$ is given by
\begin{align*}
d(\lambda)&=\inf\limits_{\boldsymbol{u}^{i}\in\mathcal{U}^{i}_{T}(x^{i})}\mathcal{L}(\boldsymbol{u}^{i},\lambda)\\
&=\sum_{i=1}^{M}\left(\lambda^{\top}\left(F^{i}x^{i}-\frac{b(\epsilon)}{M}\right)-J_{\bot}^{i}(-(H^{i})^{\top}\lambda)\right),
\end{align*}
where $\lambda\in\mathbb{R}^{N\rho}$ is the dual variable satisfying $\lambda\succeq\mathbf{0}$, $\mathcal{L}(\boldsymbol{u}^{i},\lambda)
=\sum_{i=1}^{M}J^{i}(x^{i},\boldsymbol{u}^{i})+\lambda^{\top}(\sum_{i=1}^{M}\mathfrak{g}^{i}(x^{i},\boldsymbol{u}^{i})- b(\epsilon))$ is the Lagrangian function, and $J_{\bot}^{i}(-(H^{i})^{\top}\lambda)=\sup\nolimits_{\boldsymbol{u}^{i}\in\mathcal{U}^{i}_{T}(x^{i})}\{-J^{i}(x^{i},\boldsymbol{u}^{i})-((H^{i})^{\top}\lambda)^{\top}\boldsymbol{u}^{i}\}$ is the conjugate function of $J^{i}(x^{i},\boldsymbol{u}^{i})$. Then, the dual problem of $\mathcal{P}(x)$ is given by
\begin{align}\label{dual problem}
\mathcal{DP}(x): \ \min\limits_{\lambda\succeq\mathbf{0}}\sup\limits_{\boldsymbol{u}^{i}\in\mathcal{U}^{i}_{T}(x^{i})}-\mathcal{L}(\boldsymbol{u}^{i},\lambda)\triangleq\min\limits_{\lambda\succeq\mathbf{0}}\sum\limits_{i=1}^{M}f^{i}(\lambda),
\end{align}
where $f^{i}(\lambda)=\lambda^{\top}(-F^{i}x^{i}+\frac{b(\epsilon)}{M})+J_{\bot}^{i}(-(H^{i})^{\top}\lambda)$. Since $J^{i}(x^{i},\boldsymbol{u}^{i})$ is $\mu^{i}$-strongly convex and has an $l^{i}$-Lipschitz continuous gradient, its conjugate function $J_{\bot}^{i}(-(H^{i})^{\top}\lambda)$ is $\frac{1}{l^{i}}$-strongly convex and has a $\frac{1}{\mu^{i}}$-Lipschitz continuous gradient \cite{Li2018Accelerated}. Moreover, since the supremum in the definition of $J_{\bot}^{i}(-(H^{i})^{\top}\lambda)$ is attainable and by Danskin's Theorem \cite{Bertsekas1999}, the gradient of $f^{i}(\lambda)$ is denoted by $\nabla f^{i}(\lambda)=-F^{i}x^{i}+\frac{b(\epsilon)}{M}-H^{i}\arg\max\nolimits_{\boldsymbol{u}^{i}\in\mathcal{U}^{i}_{T}(x^{i})}\{-J^{i}(x^{i},\boldsymbol{u}^{i})-((H^{i})^{\top}\lambda)^{\top}\boldsymbol{u}^{i}\}$. It is noted that $f^{i}(\lambda)$ is $\vartheta^{i}$-strongly convex and $\nabla f^{i}(\lambda)$ is $L^{i}$-Lipschitz continuous, where $\vartheta^{i}=\frac{\lambda_{\min}(H^{i}(H^{i})^{\top})}{l^{i}}$ and $L^{i}=\frac{\lambda_{\max}(H^{i}(H^{i})^{\top})}{\mu^{i}}=\frac{\|H^{i}\|^{2}_{2}}{\mu^{i}}$ \cite{Nesterov2005}. To simplify notation, we let $\vartheta=\min\limits_{1\leq i\leq M}\Big\{\frac{\lambda_{\min}(H^{i}(H^{i})^{\top})}{l^{i}}\Big\}$ and $L=\max\limits_{1\leq i\leq M}\Big\{\frac{\|H^{i}\|^{2}_{2}}{\mu^{i}}\Big\}$. In this paper, we will develop a distributed algorithm to find the saddle point $(\boldsymbol{u}^{i}_{\star},\lambda_{\star})$ of the dual problem $\mathcal{DP}(x)$.

\section{APDG Algorithm}\label{S3}
\subsection{APDG Algorithm Design}\label{S3.1}
In this subsection, we propose an asynchronous algorithm called APDG, which is summarized in Algorithm \ref{Algorithm1}.

\begin{algorithm}[t]
    \caption{The APDG Algorithm}\label{Algorithm1}
    \renewcommand{\algorithmicrequire}{\textbf{Initialize:}}
    \begin{algorithmic}[1]
        \REQUIRE For subsystem $i\in\mathcal{V}$, set the initial values $\lambda^{i}_{0}\succeq\mathbf{0}$, $z^{i}_{0}=\mathbf{0}$, $y^{i}_{0}=1$, $d^{i}_{0}=\nabla f^{i}(\lambda^{i}_{0})=\mathbf{0}$, $l^{i}_{0}=0$, $s^{i}_{0}=0$ and the local iteration index $k_{i}=0$; create local buffers $\bar{\mathcal{Z}}^{i}$, $\bar{\mathcal{Y}}^{i}$, $\bar{\mathcal{S}}^{i}$, $\bar{\mathcal{L}}^{i}$ and $\bar{\mathcal{D}}^{i}$; and send $z^{i}_{0}$,
$d^{i}_{0}$, $y^{i}_{0}$, $s^{i}_{0}$ and $l^{i}_{0}$ to its all out-neighbors $N^{i}_{\text{out}}$.\\
            \FOR {each subsystem $i$ at iteration $k_{i}$}
            \STATE Keep receiving $z^{j}_{k_{j}}$, $y^{j}_{k_{j}}$, $d^{j}_{k_{j}}$, $s^{j}_{k_{j}}$ and $l^{j}_{k_{j}}$ from its in-neighbors, $j\in N^{i}_{\text{in}}$. Then, these data are stored separately in the buffers $\bar{\mathcal{Z}}^{i}$, $\bar{\mathcal{Y}}^{i}$, $\bar{\mathcal{D}}^{i}$, $\bar{\mathcal{S}}^{i}$ and $\bar{\mathcal{L}}^{i}$.
            \STATE Read data from these buffers. When $l^{j}_{k_{j}}=0$, clear data about its neighbor $j$ in the buffers. Then update
            \begin{align}
              &w^{i}_{k_{i}+1}=\sum\nolimits_{z^{j}_{k_{j}}\in\bar{\mathcal{Z}}^{i}}a_{ij}z^{j}_{k_{j}}, \label{AA0}\\
              &y^{i}_{k_{i}+1}=\sum\nolimits_{y^{j}_{k_{j}}\in\bar{\mathcal{Y}}^{i}}a_{ij}y^{j}_{k_{j}}, \label{AA1}\\
              &\lambda^{i}_{k_{i}+1}=\frac{[w^{i}_{k_{i}+1}]_{+}}{y^{i}_{k_{i}+1}},
              \label{AA2}\\
              &\boldsymbol{u}^{i}_{k_{i}+1}={\underset{\boldsymbol{u}^{i}\in\mathcal{U}^{i}_{T}(x^{i})}{{\arg\max}}}\Big\{\!-\!J^{i}(x^{i},\boldsymbol{u}^{i})\!-\!((H^{i})^{\top}\lambda^{i}_{k_{i}+1})^{\top}\boldsymbol{u}^{i}\Big\},
              \label{AA3}\\
              &\alpha^{i}_{k_{i}+1}\!=\!\sum\nolimits_{s=s^{i}_{k_{i}}}^{\tilde{s}^{i}_{k_{i}+1}}\beta \ \text{with} \ \tilde{s}^{i}_{k_{i}+1}\!=\!\max\nolimits_{s^{j}_{k_{j}}\!\in\bar{\mathcal{S}}^{i}}\left\{s^{j}_{k_{j}}\right\},
              \label{AA5}\\
              &z^{i}_{k_{i}+1}=w^{i}_{k_{i}+1}-\alpha^{i}_{k_{i}+1}d^{i}_{k_{i}},
              \label{AA6}\\
              &d^{i}_{k_{i}+1}\!=\!\sum\nolimits_{d^{j}_{k_{j}}\!\in\bar{\mathcal{D}}^{i}}a_{ij}d^{j}_{k_{j}}\!+\!\nabla f^{i}(\lambda^{i}_{k_{i}+1})\!-\!\nabla f^{i}(\lambda^{i}_{k_{i}}).
              \label{AA7}
            \end{align}
            \STATE Let $k_{i}=k_{i}+1$ and $s^{i}_{k_{i}}=k_{i}$.
            \STATE Send $z^{i}_{k_{i}}$, $y^{i}_{k_{i}}$, $d^{i}_{k_{i}}$ and $s^{i}_{k_{i}}$ to its all out-neighbors $N^{i}_{\text{out}}$.
            \IF {the termination conditions \eqref{lemma 8 eq1} and \eqref{lemma 8 eq2} hold}
            \STATE Let $l^{i}_{k_{i}}=1$, and send $l^{i}_{k_{i}}$ to its all out-neighbors $N^{i}_{\text{out}}$.
            \STATE Output $\lambda^{i}_{k_{i}}$ and $\boldsymbol{u}^{i}_{k_{i}}$.
            \STATE \textbf{break}
            \ELSE
            \STATE Let $l^{i}_{k_{i}}=0$, and send $l^{i}_{k_{i}}$ to its all out-neighbors $N^{i}_{\text{out}}$.
            \STATE Go back to step 2.
            \ENDIF
            \ENDFOR
    \end{algorithmic}
\end{algorithm}

We now introduce the notation for our APDG algorithm. In the equations \eqref{AA0}-\eqref{AA2}, $w^{i}_{k_{i}}\in\mathbb{R}^{N\rho}$, $z^{i}_{k_{i}}\in\mathbb{R}^{N\rho}$ and $y^{i}_{k_{i}}\in\mathbb{R}$ are auxiliary variables; $\lambda^{i}_{k_{i}}$ is the local estimate of the dual variable $\lambda$ generated by subsystem $i$ at iteration $k_{i}$; and the symbol $[\bullet]_{+}$ is the projection operator, and for a given vector $v\in\mathbb{R}^{n}$, $[v]_{+}:$ $\mathbb{R}^{n}\rightarrow \mathbb{R}^{n}$, is defined by $[v]_{+}=\text{col}(\max\{0,v^{1}\},\cdots,\max\{0,v^{n}\})$. In \eqref{AA5}, $s^{i}_{k_{i}}$ is the current iteration number of subsystem $i$; $\beta$ is the fixed step-size to be designed; and $\alpha^{i}_{k_{i}}$ is the adaptive step-size at iteration $k_{i}$. In \eqref{AA6} and \eqref{AA7}, we write $\nabla f^{i}(\lambda^{i}_{k_{i}})$ for the gradient at iteration $k_{i}$, which is given by $\nabla f^{i}(\lambda^{i}_{k_{i}})=-F^{i}x^{i}+\frac{b(\epsilon)}{M}-H^{i}\boldsymbol{u}^{i}_{k_{i}}=-\big(\mathfrak{g}^{i}(x^{i},\boldsymbol{u}^{i}_{k_{i}})-\frac{b(\epsilon)}{M}\big)$; and $d^{i}_{k_{i}}$ is a proxy of the weighted average gradient.

In what follows, we provide some explanations for our APDG algorithm. First, the equations \eqref{AA0}-\eqref{AA2} are the push-sum steps to obtain the local estimate $\lambda^{i}_{k_{i}}$ of the dual variable $\lambda$. The projection operator $[\bullet]_{+}$ is used in \eqref{AA2} to ensure the local estimate $\lambda^{i}_{k_{i}}\succeq\mathbf{0}$. Then, the DMPC law can be obtained by solving \eqref{AA3}. The adaptive step-size scheme, designed in \eqref{AA5}, is utilized to reduce the convergence speed difference caused by different iterations among neighbors. The equation \eqref{AA6} is similar to the gradient descent technique; however, it differs in that we use a fixed step-size in our adaptive step-size scheme and a proxy of the weighted average gradient. Finally, in asynchronous updates, we develop the gradient tracking technique in \eqref{AA7} to update \eqref{AA6}. It is noted that \eqref{AA7} utilizes historical gradient information to track the weighted average gradient, thereby accelerating the convergence of our APDG algorithm.

Compared with the synchronous push-sum dual gradient (PSDG) algorithm \cite{Jin2021}, our APDG algorithm has the following advantages: (i) It updates and communicates asynchronously, without any synchronous waiting. Once any subsystem finishes its current updates, it can immediately start the subsequent ones, thereby avoiding idle waiting time and significantly improving the solving efficiency of the optimization problem. (ii) Our APDG algorithm does not require a central unit for clock synchronization, which significantly simplifies implementation. (iii) Most importantly, our APDG algorithm has a superior convergence rate, as stated in Theorem \ref{theorem 1}.

\begin{remark}
Our adaptive step-size scheme is inspired by \cite{Zhang2019You}, but it has several key differences. The scheme in \cite{Zhang2019You} adopts diminishing step-sizes. Specifically, diminishing step-sizes satisfy $\sum_{k_{i}=0}^{\infty}\beta^{i}_{k_{i}}=\infty$ and $\sum_{k_{i}=0}^{\infty}(\beta^{i}_{k_{i}})^{2}<\infty$ (see \cite[Assumption 1]{Zhang2019You}), which ensures the convergence of the algorithm. In contrast, our adaptive step-size scheme utilizes a fixed step-size $\beta$, as long as it satisfies the specifically designed condition (see Theorem \ref{theorem 1}). Furthermore, the adaptive step-size scheme in \cite{Zhang2019You} requires that the sum of the step-sizes of each subsystem be equal (as detailed in \cite[Section III]{Zhang2019You}), whereas our scheme does not impose this strict condition.
\end{remark}

\subsection{Reformulation of The APDG Algorithm}\label{S3.2}
In this subsection, we reformulate the updates of the APDG algorithm into an augmented system.

Define a global index as $k$. Whenever at least one local index $k_{i}$ is incremented by 1, the global index $k$ is automatically incremented by 1. (see Fig. \ref{fig1} for an intuitive representation). We write $t(k)$ for the time at the global index $k$. We let $\mathcal{T}\triangleq\{t(k)\}_{k\in\mathbb{N}}$ be the set of activation times of all subsystems. $\mathcal{T}^{i}$ is the set of activation times of subsystem $i$, $i\in\mathcal{V}$. It is evident that $\mathcal{T}^{i}\subseteq\mathcal{T}$ and $\mathcal{T}=\mathcal{T}^{1}\bigcup\mathcal{T}^{2}\bigcup\cdots\bigcup\mathcal{T}^{M}$. If $t(k)\in\mathcal{T}^{i}$, subsystem $i$ is activated at time $t(k)$. At time $t(k)$, we define $\mathcal{E}(t(k))$ as the edge set formed by the activated subsystems and their out-neighbors, and we denote it as $\mathcal{E}(k)$ for brevity. At time $t(k)$, the digraph is represented as $\mathcal{G}(t(k))=(\mathcal{V},\mathcal{E}(k))$, and we write it as $\mathcal{G}(k)$ for brevity. If $t(k)\in\mathcal{T}^{i}$ and $(i\rightarrow j)\in\mathcal{E}$, then $(i\rightarrow j)\in\mathcal{E}(k)$.

\begin{figure}[!ht]
\centerline{\includegraphics[width=8cm]{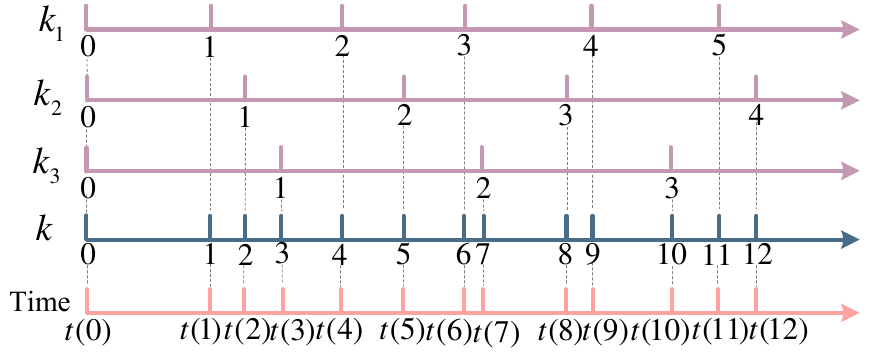}}
\caption{Iteration counts and activation times of three subsystems in the APDG algorithm.}
\label{fig1}
\end{figure}

To continue, we introduce the following assumption and lemma.
\begin{assumption}\!\!\!\cite{Zhang2019You}\label{assumption 5}
For any two consecutive activation times $t_{i}$ and $t_{i}^{+}$ of subsystem $i$, ($t_{i}$, $t_{i}^{+}\in\mathcal{T}^{i}$, $i\in\mathcal{V}$), there always exist two positive constants $\underline{\tau}$ and $\bar{\tau}$ such that $\underline{\tau}\leq|t_{i}-t_{i}^{+}|\leq\bar{\tau}$.
\end{assumption}

\begin{lemma}\!\!\!\cite{Zhang2019You}\label{lemma 1}
Under Assumptions \ref{assumption 3} and \ref{assumption 5}, the following statements hold.

(1) Each subsystem is activated at least once during $(t(k), t(k+\eta_{1}^{})]$, where $\eta_{1}^{}=(M-1)\lfloor\bar{\tau}/\underline{\tau}\rfloor+1$. Moreover, for any $k$, the union of graphs $\bigcup_{t=k}^{k+\eta_{1}}\mathcal{G}(t)$ is strongly connected.

(2) Let $\eta_{2}^{}=M\lfloor\tau/\underline{\tau}\rfloor$ and $\eta=\eta_{1}^{}+\eta_{2}^{}$. If $t(k)\in\mathcal{T}^{i}$ and $(i\rightarrow j)\in\mathcal{E}$, then at time $t(k)$, data sent by subsystem $i$ can be received by subsystem $j$ before time $t(k+\eta_{2}^{})$, and before time $t(k+\eta)$, subsystem $j$ will use these data to complete one update.

(3) For any $i,j\in\mathcal{V}$ and $k$, we have $|s^{i}(k)\!-\!s^{j}(k)|\!\leq\! M\eta$, where $s^{i}(k)$ and $s^{j}(k)$ are the values of $s^{i}_{k_{i}}$ and $s^{j}_{k_{j}}$ at time $t(k)$.
\end{lemma}

Now we begin constructing an augmented digraph to address asynchrony and communication delay. To this end, we add $\eta$ virtual subsystems to each subsystem. The total number of real and virtual subsystems is $\hat{M}=M(\eta+1)$, where $\eta$ is defined in Lemma \ref{lemma 1}. The set of virtual subsystems of subsystem $i$ is denoted by $\tilde{\mathcal{V}}^{i}=\{\mathcal{V}^{i}_{1},\mathcal{V}^{i}_{2},\cdots,\mathcal{V}^{i}_{\eta}\}$. For any time $t(k)$, we can use an augmented digraph $\mathcal{\bar{G}}(k)=(\mathcal{\bar{V}},\mathcal{\bar{E}}(k))$ to denote the communication digraph, where $\mathcal{\bar{V}}=\mathcal{V}\bigcup\tilde{\mathcal{V}}^{1}\bigcup\cdots\bigcup\tilde{\mathcal{V}}^{M}$, and $\mathcal{\bar{E}}(k)$ is the set of directed edges between real and virtual subsystems at time $t(k)$. Note that for any $i\in\mathcal{V}$, the $i$-th subsystem in $\mathcal{\bar{G}}(k)$ corresponds to the $i$-th subsystem in $\mathcal{G}$, while the $(Ms+i)$-th subsystem in $\mathcal{\bar{G}}(k)$ represents the virtual subsystem $\mathcal{V}^{i}_{s}$, $s\in\mathbb{Z}^{\eta}_{1}$. Specifically, if $t(k)\in\mathcal{T}^{i}$ and $(i\rightarrow j)\in\mathcal{E}$, it is clear that only one of the edges $(i\rightarrow\mathcal{V}^{j}_{1})$, $(i\rightarrow\mathcal{V}^{j}_{2})$, $\cdots$, $(i\rightarrow\mathcal{V}^{j}_{\eta})$ and $(i\rightarrow j)$ is contained in $\mathcal{\bar{E}}(k)$, while the edges $(\mathcal{V}^{j}_{\eta}\rightarrow\mathcal{V}^{j}_{\eta-1})$, $(\mathcal{V}^{j}_{\eta-1}\rightarrow\mathcal{V}^{j}_{\eta-2})$, $\cdots$, $(\mathcal{V}^{j}_{1}\rightarrow j)$ are always contained in $\mathcal{\bar{E}}(k)$. Obviously, $(i\rightarrow j)\in\mathcal{\bar{E}}(k)$ if there is no communication delay and asynchrony between subsystem $i$ and its out-neighbor $j$. If subsystem $i$ sends data to its out-neighbor $j$ at time $t(k)$ and suffers from delay and asynchrony, and subsystem $j$ will use data to complete the update at time $t(k+n+1)$, $n\leq \eta-1$, then $(i\rightarrow \mathcal{V}^{j}_{n})\in\mathcal{\bar{E}}(k)$.

To facilitate subsequent analysis, we define $w^{i}(t(k))$, $y^{i}(t(k))$, $\lambda^{i}(t(k))$, $\boldsymbol{u}^{i}(t(k))$, $d^{i}(t(k))$ and $z^{i}(t(k))$ as the respective values of the original variables $w^{i}_{k_{i}}$, $y^{i}_{k_{i}}$, $\lambda^{i}_{k_{i}}$, $\boldsymbol{u}^{i}_{k_{i}}$, $d^{i}_{k_{i}}$ and $z^{i}_{k_{i}}$ in Algorithm \ref{Algorithm1} at time $t(k)$. They can be abbreviated as $w^{i}(k)$, $y^{i}(k)$, $\lambda^{i}(k)$, $\boldsymbol{u}^{i}(k)$, $d^{i}(k)$ and $z^{i}(k)$. By the augmented digraph $\mathcal{\bar{G}}(k)$, we reformulate the updates in Algorithm \ref{Algorithm1} as the following augmented system:
\begin{subequations}\label{augmented system}
\begin{align}
\hat{\boldsymbol{w}}(k+1)&=(\hat{\mathcal{A}}(k)\otimes I_{N\rho\times N\rho}^{})\hat{\boldsymbol{z}}(k),
\label{augment system1}\\
\hat{\boldsymbol{y}}(k+1)&=\hat{\mathcal{A}}(k)\hat{\boldsymbol{y}}(k),\label{augment system2}\\
\lambda^{i}(k+1)&=\frac{[w^{i}(k+1)]_{+}}{y^{i}(k+1)},\label{augment system3}\\
\boldsymbol{u}^{i}(k+1)&={\underset{\boldsymbol{u}^{i}\in\mathcal{U}^{i}_{T}(x^{i})}{{\arg\max}}}\big\{-J^{i}(x^{i},\boldsymbol{u}^{i})\!-\!((H^{i})^{\top}\lambda^{i}(k+1))^{\top}\boldsymbol{u}^{i}\big\},\label{augment system4}\\
\hat{\boldsymbol{d}}(k+1)&=\!(\hat{\mathcal{A}}(k)\!\otimes\! I_{N\rho\times\! N\rho})\hat{\boldsymbol{d}}(k)\!+\!\nabla\hat{\boldsymbol{f}}(k\!+1)\!-\!\nabla\hat{\boldsymbol{f}}(k),\label{augment system5}\\
\hat{\boldsymbol{z}}(k+1)&=\hat{\boldsymbol{w}}(k+1)-\hat{\boldsymbol{g}}(k+1),\label{augment system6}
\end{align}
\end{subequations}
where $\hat{\boldsymbol{w}}(k)\!\!=\!\!\text{col}(w^{1}(k),\!\cdots\!,w^{\hat{M}}(k))$,  $\hat{\boldsymbol{z}}(k)\!\!=\!\!\text{col}(z^{1}(k),\!\cdots\!,z^{\hat{M}}(k))$,\\
 $\hat{\boldsymbol{y}}(k)=\text{col}(y^{1}(k),\cdots,y^{\hat{M}}(k))$, $\hat{\boldsymbol{d}}(k)=\text{col}(d^{1}(k),\cdots,d^{\hat{M}}(k))$,
 $\nabla\hat{\boldsymbol{f}}(k)=\text{col}(\nabla f^{1}(\lambda^{1}(k)),\cdots,\nabla f^{M}(\lambda^{M}(k)),\mathbf{0}_{(\hat{M}-M)N\rho})
 =\text{col}(\nabla f^{1}(k),\cdots,\nabla f^{M}(k),\mathbf{0}_{(\hat{M}-M)N\rho})$,
 $\hat{\boldsymbol{z}}(0)=\mathbf{0}, \hat{\boldsymbol{y}}(0)=\mathbf{1}, \hat{\boldsymbol{d}}(0)=\nabla\hat{\boldsymbol{f}}(0)=\mathbf{0}$,
 $\hat{\boldsymbol{g}}(k+1)=\text{col}(\hat{g}^{1}(k+1),\cdots,\hat{g}^{M}(k+1),\mathbf{0}_{(\hat{M}-M)N\rho})$
with
\begin{align*}
 \hat{g}^{i}(k+1)&=\left\{\begin{array}{ll}
          \sum_{s=s^{i}(k)}^{\tilde{s}^{i}(k+1)}\beta d^{i}(k), &\text{if} \ i\in\mathcal{V} \ \text{and} \ t(k)\in\mathcal{T}^{i},\\
          0, &\text{otherwise},
        \end{array}
      \right.
\end{align*}
and $\tilde{s}^{i}(k+1)=\max\nolimits_{s^{j}(k)\in\bar{\mathcal{S}}^{i}}\{s^{j}(k)\}$.
The time-varying weight matrix $\hat{\mathcal{A}}(k)=[\hat{a}_{ji}(k)]_{j,i=1}^{\hat{M}}$ in \eqref{augmented system} at time $t(k)$ is given by
\begin{align*}
\hat{a}_{ji}(k)=\left\{\begin{array}{ll}
          \frac{1}{|N^{i}_{\text{out}}|}, &\text{if} \ t(k)\in\mathcal{T}^{i}, \ i,d\in\mathcal{V}, \  j=Mn+d \\
          \ & \ \text{and} \ n=\tau_{di}(k),\\
          1, &\text{if} \ i=j, i\in\mathcal{V}\ \text{and} \ t(k)\notin\mathcal{T}^{i},\\
          1, &\text{if} \ j=i-M \ \text{and} \ i\notin\mathcal{V},\\
          0, &\text{otherwise},
        \end{array}
      \right.
\end{align*}
where $\tau_{di}(k)$ is the communication delay at time $t(k)$ from subsystem $i$ to subsystem $d$. Note that $\hat{\mathcal{A}}(k)$ is a column-stochastic matrix satisfying $\mathbf{1}^{\top}\hat{\mathcal{A}}(k)=\mathbf{1}^{\top}$ for any $k$.

The preliminary lemma is presented below.
\begin{lemma}\label{proposition 1}
Under Assumptions \ref{assumption 2}-\ref{assumption 5}, consider the sequence $\{\lambda^{i}(k+1)\}_{k\in\mathbb{N}}$ generated by \eqref{augment system3}. For any $k\in\mathbb{N}$ and $i\in\mathcal{V}$, we have that
\begin{align}\label{qiwang}
\|\lambda^{i}(k+1)\!-\!\bar{z}(k)\|_{1}\!\leq\!\pi_{1}\xi^{k}\|\lambda^{i}(1)\!-\!\bar{z}(0)\|_{1}\!+\!\pi_{1}\sum_{s=1}^{k}\xi^{k-s}\|\hat{\boldsymbol{g}}(s)\|_{1},
\end{align}
where $\bar{z}(k)=\left[\frac{1}{\hat{M}}\sum_{i=1}^{\hat{M}}z^{i}(k)\right]_{+}$, $\pi_{1}=2c_{1}$, $c_{1}=M^{M\eta}c\hat{M}$, $c=2\frac{1+\bar{a}^{-b}}{1-\bar{a}^{b}}$, $b=(M-1)(\eta_{1}+1)+M(\tau+1)$, and $\xi=(1-\bar{a}^{b})^{\frac{1}{b}}$.
\end{lemma}
\begin{proof}
The proof is provided in Appendix I.
\end{proof}

\section{Convergence Analysis and Distributed Termination Criterion Design}\label{S4}
In this section, based on the augmented system \eqref{augmented system}, we present the convergence analysis and the distributed termination criterion for the APDG algorithm.

\subsection{Convergence Analysis}\label{S4.1}
\begin{theorem}\label{theorem 1}
Under Assumptions \ref{assumption 2}-\ref{assumption 5}, consider the sequence $\{\lambda^{i}(k+1)\}_{k\in\mathbb{N}}$ generated by \eqref{augment system3}. If the fixed step-size $\beta$ satisfies
$\beta\in(0,\beta_{2})$, then $\bm{\lambda}(k+1)$ converges to $\mathbf{1}\otimes\lambda_{\star}$ at an $R$-linear (geometric) rate, where $\bm{\lambda}(k+1)=\text{col}(\lambda^{1}(k+1),\cdots,\lambda^{M}(k+1))$,
that is,
\begin{align}\label{theorem 1 eq2}
\|\bm{\lambda}(k+1)-\mathbf{1}\otimes \lambda_{\star}\|_{2}=\mathcal{O}(\delta^{k})
\end{align}
with $\delta\in(\xi,1)$ given by
\begin{align}\label{theorem 1 eq3}
\delta=\left\{\begin{array}{ll}
          1-(\vartheta-c_{9})\beta, &\text{if} \ \ \beta\in(0,\beta_{1}],  \\
          \sqrt{\mathcal{L}\beta}+\xi, &\text{if} \ \ \beta\in(\beta_{1},\beta_{2}),
        \end{array}
      \right.
\end{align}
where $\beta_{1}=\Big(\frac{\sqrt{\mathcal{L}+4(\vartheta-c_{9})(1-\xi)}-\sqrt{\mathcal{L}}}{2(\vartheta-c_{9})}\Big)^{2}$, $\beta_{2}=\frac{(1-\xi)^{2}}{\mathcal{L}}$, $\mathcal{L}=(c_{11}+c_{10}L)(c_{12}+c_{13})$, $c_{10}=\left((2+\pi_{4})\pi_{1}M+\frac{\sqrt{M}}{\hat{M}}\right)(M\eta+1)\frac{1}{c_{9}}$, $c_{11}=\pi_{1}M(M\eta+1)$, $c_{12}=c_{3}\pi_{4}^{2}(\xi^{-1}+\xi^{-2})\sqrt{N\rho}M$, $c_{13}=M\pi_{4}(c_{3}(1+\xi^{-1})+\sqrt{N\rho})$, $0<c_{9}<\vartheta$, $c_{3}=c_{2}\sqrt{\hat{M}N\rho}$, $c_{2}=c\xi^{-1}$, $c=2\frac{1+\bar{a}^{-b}}{1-\bar{a}^{b}}$, $\pi_{1}=2c_{1}$, $c_{1}=M^{M\eta}c\hat{M}$, $\xi=(1-\bar{a}^{b})^{\frac{1}{b}}\in(0,1)$, and $b=(M-1)(\eta_{1}+1)+M(\tau+1)$. Moreover, $\boldsymbol{u}^{i}(k+1)$ generated by \eqref{augment system4} also converges to the optimal solution $\boldsymbol{u}^{i}_{\star}$ at an $R$-linear rate, that is, $\|\boldsymbol{u}^{i}(k+1)-\boldsymbol{u}^{i}_{\star}\|_{2}=\mathcal{O}(\delta^{k})$.
\end{theorem}
\begin{proof}
See Appendix II.
\end{proof}
\begin{remark}
It is noted that \eqref{theorem 1 eq2} is a typical form of the $R$-linear convergence rate, with the factor $\delta$ strictly satisfying $\delta\in(\xi,1)$. By the definition of $b$, we observe that as the upper bound of communication delay $\tau$ increases, $b$ also increases. Therefore, we conclude that as $\tau$ increases, $\xi$ is monotonically increasing and always less than 1. On the other hand, by the claim (2) of Lemma \ref{lemma 1}, we find that with the increase of $\tau$, the data received by subsystems will also be delayed accordingly. This affects the convergence performance of the algorithm, but it still maintains an $R$-linear convergence rate.
\end{remark}
\begin{remark}
It can be observed from \eqref{theorem 1 eq3} that $\delta>1-(\vartheta-c_{9})\frac{(1-\xi)^{2}}{\mathcal{L}}>1-\vartheta\frac{(1-\xi)^{2}}{\mathcal{L}}>1-\frac{\vartheta}{L}$. It is known that the centralized Nesterov gradient descent algorithm \cite{Nesterov2013}, designed for unconstrained centralized optimization problems, achieves an optimal linear convergence rate of $\mathcal{O}\Big(\Big(1-\sqrt{\frac{\vartheta}{L}}\Big)^{k}\Big)$, where $0<\frac{\vartheta}{L}<1$ is the condition number. Therefore, it is clear that  $\delta>1-\sqrt{\frac{\vartheta}{L}}$, which shows that our linear rate does not violate Nesterov's optimality. However, it should be noted that our APDG algorithm is applicable to constrained distributed optimization problems over directed graphs, such as the concerned DMPC problem in this paper, which makes it more practical.
\end{remark}

\subsection{Distributed Termination Criterion}\label{S4.2}
Typically, Algorithm \ref{Algorithm1} can be terminated when its output solution satisfies certain conditions, which can prevent the algorithm from being updated infinitely and save computational costs. To this end, for Algorithm \ref{Algorithm1} we consider the following termination conditions:
\vspace{-0.25cm}
\begin{subequations}\label{termination 1}
\begin{align}
&\sum\limits_{i=1}^{M}\mathfrak{g}^{i}(x^{i},\boldsymbol{u}^{i}_{k_{i}+1})-b(\epsilon)\prec \epsilon M\mathbf{1}_{N\rho},\label{termination criteria eq1}\\
&\sum\limits_{i=1}^{M} J^{i}(x^{i},\boldsymbol{u}^{i}_{k_{i}+1})-\sum\limits_{i=1}^{M}J^{i}(x^{i},\boldsymbol{u}^{i}_{\star})\leq \epsilon_{g},\label{termination criteria eq2}
\end{align}
\end{subequations}
where $\epsilon_{g}$ is a positive constant. From the tightened global constraint \eqref{tightening constraint}, we note that \eqref{termination criteria eq1} guarantees the satisfaction of the global constraint, and \eqref{termination criteria eq2} ensures the specified suboptimality of the output solution. It should be noted that the optimizer $\boldsymbol{u}^{i}_{\star}$ in \eqref{termination criteria eq2} is not \emph{a priori} known. Additionally, the above termination criterion \eqref{termination 1} requires a central unit, which is relatively difficult to implement in a distributed setting. Therefore, we are interested in developing an easily implementable distributed termination criterion to replace the previous one \eqref{termination 1}, as presented in Proposition \ref{lemma 8}.
\begin{proposition}\label{lemma 8}
Under Assumptions \ref{assumption 2}-\ref{assumption 5}, for Algorithm \ref{Algorithm1}, there exists at least one iteration $k_{i}$ such that two consecutive updates of each subsystem $i$ ($i\in\mathcal{V}$) satisfy
\begin{subequations}\label{lemma 8 eq}
\begin{align}
&\mathfrak{g}^{i}(x^{i},\boldsymbol{u}^{i}_{k_{i}+1})-\mathfrak{g}^{i}(x^{i},\boldsymbol{u}^{i}_{k_{i}})\prec (\epsilon-\epsilon_{b})\mathbf{1}_{N\rho},\label{lemma 8 eq1}\\
&J^{i}(x^{i},\boldsymbol{u}^{i}_{k_{i}+1})\!+\!\|\nabla f^{i}(\lambda^{i}_{k_{i}})\|_{2}\|\triangle\lambda^{i}_{k_{i}}\|_{2}\!+\!f^{i}(\lambda^{i}_{k_{i}+1})\!\leq\! \frac{\epsilon_{g}}{M},\label{lemma 8 eq2}
\end{align}
\end{subequations}
where $\|\triangle\lambda^{i}_{k_{i}}\|_{2}=\|\lambda^{i}_{k_{i}+1}\!-\!\lambda^{i}_{k_{i}}\|_{2}$, and $\epsilon_{b}$ is a positive constant such that $\epsilon-\epsilon_{b}>0$ is sufficiently small. Besides, \eqref{termination criteria eq1} and \eqref{termination criteria eq2} can also be readily satisfied as long as each subsystem satisfies \eqref{lemma 8 eq1} and \eqref{lemma 8 eq2}.
\end{proposition}
\begin{proof}
See Appendix III.
\end{proof}

\begin{remark}
It is noted that \eqref{lemma 8 eq1} ensures the satisfaction of the global constraint, while \eqref{lemma 8 eq2} guarantees the specified suboptimality derived from \eqref{termination criteria eq2}. It is also noted that our distributed termination criterion \eqref{lemma 8 eq} does not require a central coordination unit and only needs two parameters,  $\epsilon_{g}$ and $\epsilon_{b}$, to be designed. Compared with the distributed termination criteria based on the finite-time average consensus algorithms \cite{Wang2017,Li2021Distributed,Wang2018Ong} or the gossip protocol \cite{Jin2021}, our proposed one \eqref{lemma 8 eq} introduces fewer designed parameters and does not require any form of coordination. This would be able to significantly simplify implementation.
\end{remark}

\section{The Iterative DMPC Approach}\label{S5}
In this section, we incorporate the proposed APDG algorithm into the iterative DMPC approach. At time $t$, we let
\begin{align}\label{output solotion}
\boldsymbol{u}^{i}_{\bar{k}_{i}}=\text{col}(u^{i}_{0},u^{i}_{1},\cdots,u^{i}_{N-1})
\end{align}
be the solution obtained from the proposed APDG algorithm, where $\bar{k}_{i}$ is the iteration index when the distributed termination criterion \eqref{lemma 8 eq} is satisfied. At time $t$, we take $u^{i}(t)=u^{i}_{0}$ as the DMPC law for subsystem $i$. We now summarize the proposed iterative DMPC approach in Algorithm \ref{Algorithm2}.

\begin{algorithm}[!http]
    \caption{The Iterative DMPC Approach}\label{Algorithm2}
    \begin{algorithmic}[1]
            \STATE At time $t$, subsystem $i$ measures its state $x^{i}(t)$.
            \STATE Subsystem $i$ calls Algorithm \ref{Algorithm1} and gets the output solution $\boldsymbol{u}^{i}_{\bar{k}_{i}}$.
            \STATE Use the first component in $\boldsymbol{u}^{i}_{\bar{k}_{i}}$ to stabilize subsystem $i$.
            \STATE Let $t=t+1$, and then return to step 1.
    \end{algorithmic}
\end{algorithm}

\begin{proposition}\label{lemma 9}
Under Assumptions \ref{assumption 1}-\ref{assumption 5}, if all $x^{i}(t)\in\mathcal{X}^{i}_{f}$ and the distributed termination criterion \eqref{lemma 8 eq} holds, then the output solution \eqref{output solotion} of Algorithm \ref{Algorithm1} satisfies
\begin{align}\label{lemma 9 eq1}
\boldsymbol{u}^{i}_{\bar{k}_{i}}
= K^{i}_{A}x^{i}(t)=\boldsymbol{u}^{i}_{\star},
\end{align}
where $K^{i}_{A}x^{i}(t)\!=\!\text{col}(K^{i}x^{i}(t), K^{i}A^{i}_{K}x^{i}(t),\!\cdots\!,K^{i}(A^{i}_{K})^{N-1}x^{i}(t))$, $A^{i}_{K}=A^{i}+B^{i}K^{i}$, and $\bar{k}_{i}=1$.
\end{proposition}
\begin{proof}
The proof is provided in Appendix IV.
\end{proof}

\begin{theorem}\label{theorem 2}
Under Assumptions \ref{assumption 1}-\ref{assumption 5}, if the initial state $x(0)$ is within the feasible region, i.e., $x(0)\in\mathcal{D}(x)$, then the iterative DMPC approach satisfies the recursive feasibility.
\end{theorem}
\begin{proof}
The proof is provided in Appendix V.
\end{proof}

\begin{theorem}\label{theorem 3}
Under Assumptions \ref{assumption 1}-\ref{assumption 5}, consider the iterative DMPC approach. Choose appropriate positive parameters $\epsilon_{g}$ and $\epsilon^{\prime}$ such that $\mathcal{X}^{i}_{\diamond}\subseteq\mathcal{X}^{i}_{f}$, where $\mathcal{X}^{i}_{\diamond}=\{x^{i}(t):\|x^{i}(t)\|_{Q^{i}}^{2}\leq\epsilon_{g}+\epsilon^{\prime}\}$. If the initial state $x(0)$ is in the feasible region, i.e., $x(0)\in\mathcal{D}(x)$, then the following results hold:\\
 (1) there exists a time $t^{\prime}$ such that the states of all subsystems can enter their respective terminal sets, i.e., for all $i\in\mathcal{V}$, $x^{i}(t^{\prime})\in\mathcal{X}^{i}_{f}$;\\
 (2) given that (1) holds, the closed-loop system is asymptotically stable at the equilibrium point.
\end{theorem}
\begin{proof}
The proof can be followed from \cite{Wang2017}, along with Propositions \ref{lemma 8} and \ref{lemma 9}. We thus omit the details.
\end{proof}

\section{Numerical Simulation} \label{S6}
In this section, we utilize the water tanks system from \cite{Wang2018Ong} to validate the effectiveness of the proposed iterative DMPC approach, whose dynamics are expressed as follows:
\begin{align*}
x^{i}(t+1)=\begin{bmatrix}
             0.8750&0.1250 \\
             0.1250& 0.8047
           \end{bmatrix}x^{i}(t)+\begin{bmatrix}
                                   0.3\\
                                   0
                                 \end{bmatrix}u^{i}(t), \ i\in\mathbb{Z}^{4}_{1},
\end{align*}
where $x^{i}(t)=[x^{i1}(t),x^{i2}(t)]^{\top}$. The initial states of four subsystems are given as $x^{1}(0)=[-1.8,2]^{\top}$, $x^{2}(0)=[2,-0.8]^{\top}$, $x^{3}(0)=[-1,1]^{\top}$, and $x^{4}(0)=[-0.85,0.85]^{\top}$, respectively. The local input and state constraints are given as $\mathcal{X}^{i}=\{x^{i}(t):-2\leq x^{i1}(t)\leq2,-2\leq x^{i2}(t)\leq2\}$ and $\mathcal{U}^{i}=\{u^{i}(t):-1\leq u^{i}(t)\leq1\}$, respectively, where $i=1,2,3,4$. The global constraint is set as $|\sum_{i=1}^{4}u^{i}(t)|\leq1.5$. For given $Q^{i}=5I$ and $R^{i}=1$ for $i=1,2,3,4$, by solving the discrete-time ARE, $K^{i}$ and $P^{i}$ are calculated as $K^{i}=\begin{bmatrix}-1.4110& -0.6099\end{bmatrix}$ and $P^{i}=[9.5229 \ 3.2122; 3.2122 \ 14.4820]$. The prediction horizon and the fixed step-size are chosen as $N=8$ and $\beta=0.08$. And other parameters are set as $\gamma=0.01$, $\epsilon=0.0005$, $\epsilon_{b}=0.0001$ and $\epsilon_{g}=0.0005$. The edge set of the communication digraph is given by $\mathcal{E}=\{(1\rightarrow 2),(3\rightarrow2),(2\rightarrow4),(4\rightarrow1),(4\rightarrow3),(1\rightarrow3),(3\rightarrow1)\}$. The communication delay is set as $\tau_{ij}=0.0661s$.

\begin{figure}[tbph]
\centerline{\includegraphics[width=8cm]{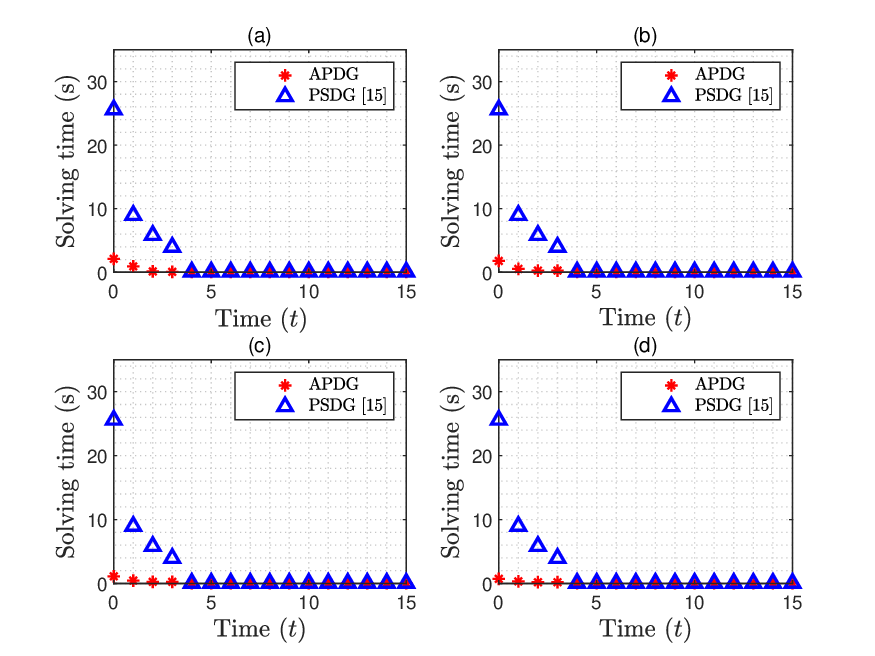}}
\caption{Solving time of the dual problem at each time $t$.}
\label{fig3}
\end{figure}

\begin{figure}[tbph]
\centerline{\includegraphics[width=8cm]{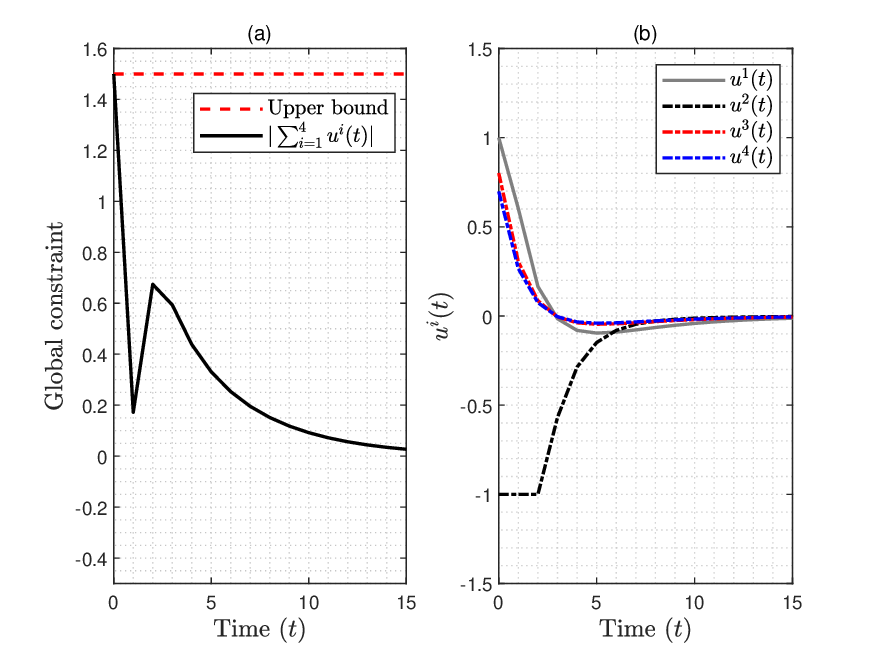}}
\caption{(a) Evolution of the global constraint. (b) Input trajectories of four subsystems.}
\label{fig4}
\end{figure}

The proposed algorithm is implemented by utilizing the YALMIP toolbox in Matlab 2021b on Windows 10 with Intel(R) Core(TM) i5-10200H CPU and 16 GB RAM. We assume that the computational capability of subsystem 1 is twice that of subsystems 2 and 3, and three times that of subsystem 4. To demonstrate the advantages of our APDG algorithm, we conduct numerical comparisons with the PSDG algorithm \cite{Jin2021} using identical settings, including parameters, communication delay, system computational capability, and initial conditions. For a fair comparison, both algorithms use the same distributed termination criterion in \eqref{lemma 8 eq}.

\begin{figure}[tbph]
\centerline{\includegraphics[width=8cm]{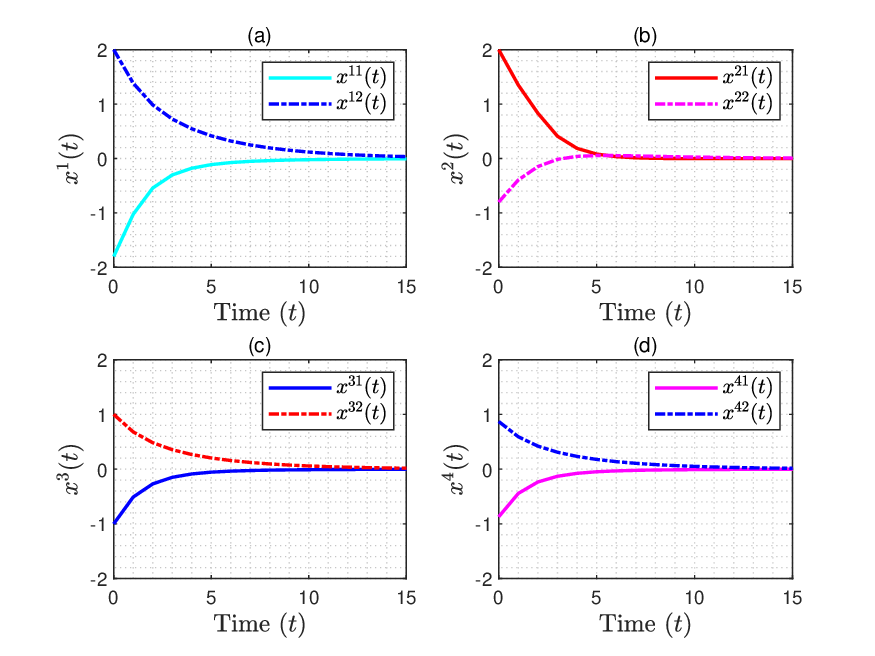}}
\caption{State trajectories of four subsystems.}
\label{fig5}
\end{figure}

Fig. \ref{fig3} shows the solving time of the dual problem under two algorithms. Fig. \ref{fig4} depicts the evolution of the global constraint and the input trajectory of each subsystem under our APDG algorithm. Fig. \ref{fig5} depicts the state trajectories of each subsystem under our APDG algorithm. Notably, since both algorithms converge to solutions with the same level of suboptimality under the same distributed termination criterion \eqref{lemma 8 eq}, the corresponding trajectories of the PSDG algorithm \cite{Jin2021} are omitted in Figs. \ref{fig4} and \ref{fig5}. As shown in Fig. \ref{fig3}, our APDG algorithm achieves a shorter time in solving the dual problem compared to the PSDG algorithm \cite{Jin2021}. It can be observed from Fig. \ref{fig3} that after $t\geq4$, two algorithms have almost the same solving time (requiring only one iteration to solve the dual problem), which validates  the conclusion in Proposition \ref{lemma 9}. In addition, it can be observed from Figs. \ref{fig4} and \ref{fig5} that our APDG algorithm not only ensures the stability of the closed-loop system but also satisfies both local input and state constraints, as well as the global constraint.

\section{Conclusion}\label{S7}
In this paper, we developed an APDG-based iterative DMPC approach for distributed discrete-time linear systems with both local and global constraints over directed communication networks. We proved that the proposed APDG algorithm achieves an $R$-linear convergence rate. We also developed a distributed termination criterion for the APDG algorithm, which ensures the specified suboptimality of the solution and reduces computational overhead. Moreover, the recursive feasibility of the proposed iterative DMPC approach and the stability of the resulting closed-loop system were established. Finally, a numerical example demonstrated that our proposed iterative DMPC approach achieves a shorter solving time while ensuring satisfactory system performance.

\section*{Appendix I: Proof of Lemma \ref{proposition 1}}
\subsection*{A. Preliminaries}
To facilitate the proof of Lemma \ref{proposition 1}, we first present some supporting lemmas.
\begin{lemma}\label{lemma 2}
Under Assumptions 2-4, the following claims hold.

(1) There exists a stochastic-vector $\phi(k)\!=\!\text{col}(\phi^{1}(k),\!\cdots\!,\phi^{\hat{M}}(k))$ satisfying $\mathbf{1}^{\top}\phi(k)=1$, such that for any $i,j\in\mathbb{Z}^{\hat{M}}_{1}$ and $k\geq s\geq0$,
      \begin{align*}
        \left|[\Phi(k,s)]_{ij}-\phi^{i}(k)\right|\leq c\xi^{k-s},
      \end{align*}
      where $\Phi(k,s)$ is the transition matrix satisfying $\Phi(k,s)=\hat{\mathcal{A}}(k)\hat{\mathcal{A}}(k-1)\cdots\hat{\mathcal{A}}(s)$, $c=2\frac{1+\bar{a}^{-b}}{1-\bar{a}^{b}}$, $b=(M-1)(\eta_{1}+1)+M(\tau+1)$, $\xi=(1-\bar{a}^{b})^{\frac{1}{b}}\in(0,1)$, and $\eta_{1}$ is defined in Lemma 1.

(2) Furthermore, for any $i\!\in\!\mathcal{V}$ and $k\!\in\!\mathbb{N}$,  $\sum_{j=1}^{M}[\Phi(k,0)]_{ij}\!\geq\! M^{-M\eta}$, where $\eta$ is defined in Lemma 1.
\end{lemma}
\begin{proof}
The proof of claim (1) follows that of [18, Lemma 5]. The proof of claim (2) is identical to that of [17, Lemma 2 b)].
\end{proof}

\begin{lemma}\label{lemma 3}
For any two vectors $a(k)=\text{col}(a^{1}(k),\cdots,a^{n}(k))$ and $b(k)=\text{col}(b^{1}(k),\cdots,b^{n}(k))$, it holds that
\setcounter{equation}{37}
\begin{align}\label{lemma eq1}
\left\|[a(k)]_{+}-\sigma_{1}[b(k)]_{+}\right\|_{1}\leq\left\|a(k)-\sigma_{1}b(k)\right\|_{1},
\end{align}
where $\sigma_{1}>0$ is a constant.
\end{lemma}
\begin{proof}
We consider the following two cases.

\textbf{Case 1}: Fix $a^{i}(k)\leq 0$ and consider any $b^{i}(k)$. First, for $b^{i}(k)\leq 0$, it follows from the definition of the projection operator that
\begin{align*}
 \left|[a^{i}(k)]_{+}-\sigma_{1}[b^{i}(k)]_{+}\right|=0\leq|a^{i}(k)-\sigma_{1}b^{i}(k)|.
\end{align*}
Then, for $b^{i}(k)>0$, we have that
\begin{align*}
\left|[a^{i}(k)]_{+}-\sigma_{1}[b^{i}(k)]_{+}\right|=\left|-\sigma_{1}b^{i}(k)\right|\leq\left|a^{i}(k)-\sigma_{1}b^{i}(k)\right|.
\end{align*}

\textbf{Case 2}: Fix $a^{i}(k)>0$ and consider any $b^{i}(k)$. When $b^{i}(k)\leq 0$, we have that
\begin{align*}
\left|[a^{i}(k)]_{+}-\sigma_{1}[b^{i}(k)]_{+}\right|=\left|a^{i}(k)\right|\leq\left|a^{i}(k)-\sigma_{1}b^{i}(k)\right|.
\end{align*}
Another one, when $b^{i}(k)>0$, it yields that
\begin{align*}
\left|[a^{i}(k)]_{+}-\sigma_{1}[b^{i}(k)]_{+}\right|=\left|a^{i}(k)-\sigma_{1}b^{i}(k)\right|.
\end{align*}
Summarizing the above two cases, for any $a^{i}(k)$ and $b^{i}(k)$, it is easy to obtain that
\begin{align*}
 \left|[a^{i}(k)]_{+}-\sigma_{1}[b^{i}(k)]_{+}\right|\leq\left|a^{i}(k)-\sigma_{1}b^{i}(k)\right|.
\end{align*}
Finally, we obtain that
\begin{align*}
\nonumber&\left\|[a(k)]_{+}-\sigma_{1}[b(k)]_{+}\right\|_{1}=\sum_{i=1}^{n}\left|[a^{i}(k)]_{+}-\sigma_{1}[b^{i}(k)]_{+}\right|\\
&\leq\sum_{i=1}^{n}\left|a^{i}(k)-\sigma_{1}b^{i}(k)\right|=\|a(k)-\sigma_{1}b(k)\|_{1},
\end{align*}
which ends the proof.
\end{proof}
\subsection*{B. Proof of Lemma \ref{proposition 1}}
By \eqref{augment system1} and \eqref{augment system6}, for any $k\in\mathbb{N}$, it holds that
\begin{align}\label{proposition eq2}
\nonumber&\hat{\boldsymbol{z}}(k+1)=(\hat{\mathcal{A}}(k)\otimes I)\hat{\boldsymbol{z}}(k)-\hat{\boldsymbol{g}}(k+1)\\
\nonumber&=(\hat{\mathcal{A}}(k)\otimes I)(\hat{\mathcal{A}}(k\!-\!1)\otimes I)\hat{\boldsymbol{z}}(k-1)\\
\nonumber&\quad-(\hat{\mathcal{A}}(k)\otimes I)\hat{\boldsymbol{g}}(k)-\hat{\boldsymbol{g}}(k+1)\\
\nonumber&=(\Phi(k,k-1)\otimes I)\hat{\boldsymbol{z}}(k-1)\\
\nonumber&\quad-(\Phi(k,k)\otimes I)\hat{\boldsymbol{g}}(k)-\hat{\boldsymbol{g}}(k+1)\\
\nonumber&=(\Phi(k,0)\otimes I)\hat{\boldsymbol{z}}(0)-(\Phi(k,1)\otimes I)\hat{\boldsymbol{g}}(1)\\
\nonumber&\quad-(\Phi(k,2)\otimes I)\hat{\boldsymbol{g}}(2)-\cdots-(\Phi(k,k)\otimes I)\hat{\boldsymbol{g}}(k)-\hat{\boldsymbol{g}}(k+1)\\
&=(\Phi(k,0)\otimes I)\hat{\boldsymbol{z}}(0)-\sum_{s=1}^{k}(\Phi(k,s)\otimes I)\hat{\boldsymbol{g}}(s)-\hat{\boldsymbol{g}}(k+1),
\end{align}
where the third equation uses $(A\otimes I)(B\otimes I)=(AB)\otimes I$ and $\Phi(k,s)=\hat{\mathcal{A}}(k)\hat{\mathcal{A}}(k-1)\cdots\hat{\mathcal{A}}(s)$, and the last equation can be inductively obtained.
Premultiplying both sides of \eqref{proposition eq2} by the matrix $\hat{\mathcal{A}}(k+1)\otimes I$ gives
\begin{align}\label{proposition eq3}
\nonumber&(\hat{\mathcal{A}}(k+1)\otimes I)\hat{\boldsymbol{z}}(k+1)\\
&=(\Phi(k+1,0)\otimes I)\hat{\boldsymbol{z}}(0)-\sum_{s=1}^{k+1}(\Phi(k+1,s)\otimes I)\hat{\boldsymbol{g}}(s).
\end{align}
Then, premultiplying both sides of \eqref{proposition eq3} by the matrix $(\phi(k+1)\mathbf{1}^{\top})\otimes I$ yields
\begin{align}\label{proposition eq4}
\nonumber&((\phi(k+1)\mathbf{1}^{\top})\otimes I)\hat{\boldsymbol{z}}(k+1)\\
&=((\phi(k+1)\mathbf{1}^{\top})\otimes I)\hat{\boldsymbol{z}}(0)-\sum_{s=1}^{k+1}((\phi(k +1)\mathbf{1}^{\top})\otimes I)\hat{\boldsymbol{g}}(s),
\end{align}
where the above equation utilizes $(A\otimes I)(B\otimes I)=(AB)\otimes I$ and $\mathbf{1}^{\top}\hat{\mathcal{A}}(k)=\mathbf{1}^{\top}$. Subtracting equation \eqref{proposition eq4} from equation \eqref{proposition eq3} yields that
\begin{align}\label{proposition eq5}
\nonumber&(\hat{\mathcal{A}}(k+1)\otimes I)\hat{\boldsymbol{z}}(k+1)=\mathcal{B}(k+1,0)\hat{\boldsymbol{z}}(0)\\
&-\sum_{s=1}^{k+1}\mathcal{B}(k+1,s)\hat{\boldsymbol{g}}(s)+((\phi(k+1)\mathbf{1}^{\top})\otimes I)\hat{\boldsymbol{z}}(k+1),
\end{align}
where $\mathcal{B}(k+1,s)=\mathcal{W}(k+1,s)\otimes I$, $\mathcal{W}(k+1,s)=\Phi(k+1,s)-\phi(k+1)\mathbf{1}^{\top}$, and the relation $(A\otimes I)-(B\otimes I)=(A-B)\otimes I$ is utilized.

It follows from \eqref{augment system1} and \eqref{proposition eq5} that
\begin{align}\label{proposition eq6}
\nonumber&\hat{\boldsymbol{w}}(k+1)=(\hat{\mathcal{A}}(k)\otimes I)\hat{\boldsymbol{z}}(k)\\
&=\mathcal{B}(k,0)\hat{\boldsymbol{z}}(0)-\sum_{s=1}^{k}\mathcal{B}(k,s)\hat{\boldsymbol{g}}(s)+((\phi(k)\mathbf{1}^{\top})\otimes I)\hat{\boldsymbol{z}}(k).
\end{align}
Moreover, for $\hat{\boldsymbol{y}}(k+1)$ in \eqref{augment system2}, it can be inductively obtained as
\begin{align}\label{proposition eq7}
\nonumber\hat{\boldsymbol{y}}(k+1)&=\Phi(k,0)\hat{\boldsymbol{y}}(0)=[\phi(k)\mathbf{1}^{\top}+\mathcal{W}(k,0)]\hat{\boldsymbol{y}}(0)\\
&=\phi(k)\hat{M}+\mathcal{W}(k,0)\mathbf{1}.
\end{align}
By combining \eqref{augment system3}, \eqref{proposition eq6} and \eqref{proposition eq7}, we get
\begin{align}\label{proposition eq8}
\nonumber&\lambda^{i}(k+1)=\frac{[w^{i}(k+1)]_{+}}{y^{i}(k+1)}\\
\nonumber&=\frac{1}{[\phi(k)\hat{M}+\mathcal{W}(k,0)\mathbf{1}]_{i}}\bigg[\Big([\mathcal{W}(k,0)]_{i,:}\otimes I\Big)\hat{\boldsymbol{z}}(0)-\sum_{s=1}^{k}\\
&\quad\Big([\mathcal{W}(k,s)]_{i,:}\otimes I\Big)\hat{\boldsymbol{g}}(s)+\Big([\phi(k)\mathbf{1}^{\top}]_{i,:}\otimes I\Big)\hat{\boldsymbol{z}}(k)\bigg]_{+},
\end{align}
where $[\phi(k)\hat{M}+\mathcal{W}(k,0)\mathbf{1}]_{i}$ denotes the $i$-th element of the vector $\phi(k)\hat{M}+\mathcal{W}(k,0)\mathbf{1}$. We now calculate the denominator of \eqref{proposition eq8}
\begin{align}\label{proposition eq11}
\nonumber&[\phi(k)\hat{M}+\mathcal{W}(k,0)\mathbf{1}]_{i}=[\phi(k)\mathbf{1}^{\top}\mathbf{1}+\mathcal{W}(k,0)\mathbf{1}]_{i}\\
&=[\Phi(k,0)\mathbf{1}]_{i}=\sum_{j=1}^{\hat{M}}[\Phi(k,0)]_{ij}\geq\sum_{j=1}^{M}[\Phi(k,0)]_{ij}\geq M^{-M\eta},
\end{align}
where the first inequality follows from the fact that all elements of $\Phi(k,0)$ are nonnegative real numbers, and the second inequality is derived from the claim (2) of Lemma \ref{lemma 2}. For $\bar{z}(k)$, we have that
\begin{align}\label{proposition eq9}
\bar{z}(k)=\left[\frac{1}{\hat{M}}\sum_{i=1}^{\hat{M}}z^{i}(k)\right]_{+}=\left[\frac{1}{\hat{M}}(\mathbf{1}^{\top}\otimes I)\hat{\boldsymbol{z}}(k)\right]_{+}.
\end{align}
It follows from \eqref{proposition eq8} and \eqref{proposition eq9} that
\begin{align}\label{proposition eq10}
\nonumber&\|\lambda^{i}(k+1)-\bar{z}(k)\|_{1}\\
\nonumber&=\bigg\|\frac{1}{[\phi(k)\hat{M}+\mathcal{W}(k,0)\mathbf{1}]_{i}}\bigg[\Big([\mathcal{W}(k,0)]_{i,:}\otimes I\Big)\hat{\boldsymbol{z}}(0)\\
\nonumber&\quad-\sum_{s=1}^{k}\Big([\mathcal{W}(k,s)]_{i,:}\otimes I\Big)\hat{\boldsymbol{g}}(s)\!+\!\Big([\phi(k)\mathbf{1}^{\top}]_{i,:}\otimes I\Big)\hat{\boldsymbol{z}}(k)\bigg]_{+}\\
\nonumber&\quad-\frac{[\phi(k)\hat{M}+\mathcal{W}(k,0)\mathbf{1}]_{i}[(\mathbf{1}^{\top}\otimes I)\hat{\boldsymbol{z}}(k)]_{+}}{\hat{M}[\phi(k)\hat{M}+\mathcal{W}(k,0)\mathbf{1}]_{i}}\bigg\|_{1}\\
\nonumber&\leq\bigg\|\frac{\hat{M}\Big(\big([\mathcal{W}(k,0)]_{i,:}\otimes I\big)\hat{\boldsymbol{z}}(0)
-\sum_{s=1}^{k}\big([\mathcal{W}(k,s)]_{i,:}\otimes I\big)\hat{\boldsymbol{g}}(s)\Big)}{\hat{M}[\phi(k)\hat{M}+\mathcal{W}(k,0)\mathbf{1}]_{i}}\\
\nonumber&\quad-\frac{[\mathcal{W}(k,0)\mathbf{1}]_{i}(\mathbf{1}^{\top}\otimes I)\hat{\boldsymbol{z}}(k)}{\hat{M}[\phi(k)\hat{M}+\mathcal{W}(k,0)\mathbf{1}]_{i}}\bigg\|_{1}\\
\nonumber&\leq\bigg\|\frac{\big([\mathcal{W}(k,0)]_{i,:}\otimes I\big)\hat{\boldsymbol{z}}(0)
\!-\!\sum_{s=1}^{k}\big([\mathcal{W}(k,s)]_{i,:}\otimes I\big)\hat{\boldsymbol{g}}(s)}{[\phi(k)\hat{M}+\mathcal{W}(k,0)\mathbf{1}]_{i}}\bigg\|_{1}\\
&\quad+\bigg\|\frac{[\mathcal{W}(k,0)\mathbf{1}]_{i}(\mathbf{1}^{\top}\otimes I)\hat{\boldsymbol{z}}(k)}{\hat{M}[\phi(k)\hat{M}+\mathcal{W}(k,0)\mathbf{1}]_{i}}\bigg\|_{1},
\end{align}
where the first inequality is derived from Lemma \ref{lemma 3} and $\left([\phi(k)\mathbf{1}^{\top}]_{i,:}\otimes I\right)\hat{\boldsymbol{z}}(k)=[\phi(k)]_{i}(\mathbf{1}^{\top}\otimes I)\hat{\boldsymbol{z}}(k)$, and the last inequality is obtained by employing the triangle inequality. By \eqref{proposition eq11}, the inequality \eqref{proposition eq10} can be rewritten as
\begin{align}\label{proposition eq12}
\nonumber&\|\lambda^{i}(k+1)-\bar{z}(k)\|_{1}\\
\nonumber&\leq\frac{M^{M\eta}}{\hat{M}}\bigg(\Big\|\hat{M}([\mathcal{W}(k,0)]_{i,:}\otimes I)\hat{\boldsymbol{z}}(0)\Big\|_{1}+\Big\|\hat{M}\sum_{s=1}^{k}([\mathcal{W}(k,s)]_{i,:}\\
\nonumber&\quad\otimes I)\hat{\boldsymbol{g}}(s)\Big\|_{1}+\Big\|[\mathcal{W}(k,0)\mathbf{1}]_{i}(\mathbf{1}^{\top}\otimes I)\hat{\boldsymbol{z}}(k)\Big\|_{1}\bigg)\\
\nonumber&\leq M^{M\eta}\bigg(\Big\|([\mathcal{W}(k,0)]_{i,:}\otimes I)\Big\|_{1}\Big\|\hat{\boldsymbol{z}}(0)\Big\|_{1}+\sum_{s=1}^{k}\Big\|[\mathcal{W}(k,s)]_{i,:}\\
\nonumber&\quad\otimes I\Big\|_{1}\Big\|\hat{\boldsymbol{g}}(s)\Big\|_{1}+\max\limits_{1\leq j\leq\hat{M}}\Big|[\mathcal{W}(k,0)]_{ij}\Big|\Big\|(\mathbf{1}^{\top}\otimes I)\hat{\boldsymbol{z}}(k)\Big\|_{1}\bigg)\\
&\leq c_{1}\bigg(\xi^{k}\|\hat{\boldsymbol{z}}(0)\|_{1}+\sum_{s=1}^{k}\xi^{k-s}\|\hat{\boldsymbol{g}}(s)\|_{1}+\xi^{k}\|(\mathbf{1}^{\top}\otimes I)\hat{\boldsymbol{z}}(k)\|_{1}\bigg),
\end{align}
where $c_{1}=M^{M\eta}c\hat{M}$, the second inequality uses the relation $\|Ax\|_{1}\leq\|A\|_{1}\|x\|_{1}$ with matrix $A$ and vector $x$, and the last inequality follows from $\|[\mathcal{W}(k,s)]_{i,:}\otimes I\|_{1}=\|[\mathcal{W}(k,s)]_{i,:}\|_{1}=\max\limits_{1\leq j\leq \hat{M}}|[\mathcal{W}(k,s)]_{ij}|\leq\|\mathcal{W}(k,s)\|_{1}=\|\Phi(k,s)-\phi(k)\mathbf{1}^{\top}\|_{1}\leq c\hat{M}\xi^{k-s}$ (this relation can be obtained by the claim (1) of Lemma \ref{lemma 2}).

Premultiplying both sides of \eqref{proposition eq3} by $\mathbf{1}^{\top}\otimes I$, and combining $\mathbf{1}^{\top}\hat{\mathcal{A}}(k)=\mathbf{1}^{\top}$ and $(A\otimes I)(B\otimes I)=(AB)\otimes I$, we obtain
\begin{align}\label{proposition eq13}
(\mathbf{1}^{\top}\otimes I)\hat{\boldsymbol{z}}(k)=(\mathbf{1}^{\top}\otimes I)\hat{\boldsymbol{z}}(0)-\sum_{s=1}^{k}(\mathbf{1}^{\top}\otimes I)\hat{\boldsymbol{g}}(s).
\end{align}
Taking 1-norm on both sides of equation \eqref{proposition eq13} results in
\begin{align}\label{proposition eq14}
\|(\mathbf{1}^{\top}\otimes I)\hat{\boldsymbol{z}}(k)\|_{1}\leq\|\hat{\boldsymbol{z}}(0)\|_{1}+\sum_{s=1}^{k}\|\hat{\boldsymbol{g}}(s)\|_{1},
\end{align}
where the triangle inequality, $\|Ax\|_{1}\leq\|A\|_{1}\|x\|_{1}$ and $\|\mathbf{1}^{\top}\otimes I\|_{1}=1$ are utilized. Substituting \eqref{proposition eq14} into \eqref{proposition eq12} yields \eqref{qiwang} as expected. The proof is thus completed.
\section*{Appendix II: Proof of Theorem \ref{theorem 1}}
We prove Theorem \ref{theorem 1} in the following three parts.

\textcolor[rgb]{0.15294,0.25098,0.5451}{\textbf{\emph{Part I}.}} \emph{Bounding consensus, gradient tracking, and optimization errors:}

At the global index $k$, we define the consensus, gradient tracking, and optimization errors as follows:
\begin{subequations}
\begin{align}
e_{\lambda}(k)&=\bm{\lambda}(k+1)-\mathbf{1}\otimes \bar{z}(k),\label{theorem 1 eq4}\\
e_{d}(k)&=d^{i}(k)-\overline{\nabla f}(k), \ e_{o}(k)=\boldsymbol{d}(k), \label{theorem 1 eq5}\\
e_{z}(k)&=\mathbf{1}\otimes\bar{z}(k)-\bm{\lambda}_{\star},\label{theorem 1 eq6}
\end{align}
\end{subequations}
where $i\in\mathcal{V}$, $\bm{\lambda}(k+1)=\text{col}(\lambda^{1}(k+1),\cdots,\lambda^{M}(k+1))$, $\bar{z}(k)=\left[\frac{1}{\hat{M}}\sum_{i=1}^{\hat{M}} z^{i}(k)\right]_{+}$, $\overline{\nabla f}(k)=\frac{y^{i}(k)}{\hat{M}}\sum_{i=1}^{\hat{M}}\nabla f^{i}(k)$, $\boldsymbol{d}(k)=\text{col}(d^{1}(k),\cdots,d^{M}(k))$, and $\bm{\lambda}_{\star}=\text{col}(\lambda_{\star},\cdots,\lambda_{\star})=\mathbf{1}\otimes \lambda_{\star}$.


We now present the following lemma that establishes bounds for the above error quantities \eqref{theorem 1 eq4}-\eqref{theorem 1 eq6}.
\begin{lemma}\label{proposition 2}
Suppose that Assumptions \ref{assumption 2}-\ref{assumption 5} hold. The error quantities defined in \eqref{theorem 1 eq4}-\eqref{theorem 1 eq6} satisfy the following relations:
\begin{subequations}\label{proposition 2 eq}
\begin{align}
\|e_{\lambda}(k+1)\|_{1}\leq&\pi_{2}\xi^{k}\|e_{\lambda}(0)\|_{1}+\pi_{3}\sum\limits_{s=0}^{k}\xi^{k-s}\|e_{o}(s)\|_{1},\label{proposition 2 eq1}\\
\|e_{d}(k+1)\|_{1}\leq &\varpi\xi^{k}\!+\!c_{4}\sum\limits_{s=0}^{k}\xi^{k-s}(\|e_{\lambda}(s)\|_{1}+\|e_{z}(s)\|_{2}),\label{proposition 2 eq2}\\
\nonumber\|e_{o}(k+1)\|_{1}\leq& \varpi M\xi^{k}+(c_{4}M+c_{5})\sum\limits_{s=0}^{k}\xi^{k-s}\|e_{\lambda}(s)\|_{1}\\
+(c_{4}M+&c_{5})\sum\limits_{s=0}^{k}\xi^{k-s}\|e_{z}(s)\|_{2}\!+\!c_{5}\|e_{d}(k)\|_{1},\label{proposition 2 eq3}\\
\|e_{z}(k+1)\|_{2}\leq&c_{6}\|e_{o}(k)\|_{1}\!+\!c_{7}\|e_{\lambda}(k)\|_{1}+c_{8}\|e_{z}(k)\|_{2},\label{proposition 2 eq4}
\end{align}
\end{subequations}
where $\pi_{2}=\pi_{1}\xi$, $\pi_{3}=\pi_{1}M(M\eta+1)\beta$, $\varpi=c\|\hat{\boldsymbol{d}}(0)\|_{1}+c_{3}L\|\bm{\lambda}(1)-\bm{\lambda}(0)\|_{2}$,  $c_{4}=c_{3}\pi_{4}(1+\xi^{-1})L\beta$, $c_{3}=c_{2}\sqrt{\hat{M}N\rho}$, $c_{2}=c\xi^{-1}$, $c_{5}=\sqrt{N\rho}M\pi_{4}L\beta$, $c_{6}=\frac{\sqrt{M}}{\hat{M}}(M\eta+1)\beta$, $c_{7}=2+\pi_{4}-\beta L$, $\pi_{4}>1$, $c_{8}=1-\beta L$,
and $\beta<\frac{1}{L}$.
\end{lemma}
\begin{proof}
The proof can be completed by using Lemma \ref{proposition 1} and the claim (3) of Lemma \ref{lemma 1}; we thus omit the details to save space.
\end{proof}

\textcolor[rgb]{0.15294,0.25098,0.5451}{\textbf{\emph{Part II}.}} \emph{Construction of the augmented error system:}

We introduce two lemmas to construct the augmented error system.
\begin{lemma}\label{lemma 4}
Consider the vector sequence $\{\omega(k)\}_{k=0}^{\infty}$ with $\omega(k)\in\mathbb{R}^{n}$. Define
\begin{align*}
\|\omega\|_{p}^{\delta,N}=\max\limits_{k=0,\cdots,N}\frac{\|\omega(k)\|_{p}}{\delta^{k}}, \ \|\omega\|_{p}^{\delta}=\sup\limits_{ k=0,\cdots,\infty}\frac{\|\omega(k)\|_{p}}{\delta^{k}},
\end{align*}
where $p$ is a constant that equals $1$ or $2$, $N\in\mathbb{N}$, and $\delta\in(\xi,1)$ is the linear rate parameter. If $\|\omega\|_{p}^{\delta}$ is bounded by some positive constant $W$, then $\|\omega(k)\|_{p}$ converges at an $R$-linear rate of $\mathcal{O}(\delta^{k})$.
\end{lemma}
\begin{proof}
By the definition of $\|\omega\|_{p}^{\delta}$, we have that $\|\omega\|_{p}^{\delta}=\sup\limits_{k=0,\cdots,\infty}\frac{\|\omega(k)\|_{p}}{\delta^{k}}\leq W$, which implies that $\|\omega(k)\|_{p}\leq W\delta^{k}$ for any $k$, i.e., $\|\omega(k)\|_{p}=\mathcal{O}(\delta^{k})$. The proof is thus completed.
\end{proof}

\begin{lemma}\label{lemma 5}
Given the vector sequences $\{\omega(k)\}_{k=0}^{\infty}$ and  $\{\upsilon_{i}(k)\}_{k=0}^{\infty}$, $i=1,2,3$, the following claims hold.

(1) If for $0\leq k\leq N$ and $N\in\mathbb{N}$, $\|\omega(k+1)\|_{1}\leq\pi_{5}\xi^{k}+\sum_{i=1}^{2}r_{i}\sum_{s=0}^{k}\xi^{k-s}\|\upsilon_{i}(s)\|_{i}+r_{3}\|\upsilon_{3}(k)\|_{1}$, where $\pi_{5}\geq0$ and $r_{i}\geq0$ for $i=1,2,3$, then we obtain that $\|\omega\|_{1}^{\delta,N+1}\leq\frac{\pi_{5}}{\delta}+\sum_{i=1}^{2}\frac{r_{i}}{\delta-\xi}\|\upsilon_{i}\|_{i}^{\delta,N}+\frac{r_{3}}{\delta}\|\upsilon_{3}\|_{1}^{\delta,N}+\|\omega(0)\|_{1}$, where $\delta\in(\xi,1)$, and $\xi\in(0,1)$ is defined in Theorem \ref{theorem 1}.

(2) If for $0\leq k\leq N$ and $N\in\mathbb{N}$, $\|\omega(k+1)\|_{2}\leq\pi_{7}\|\omega(k)\|_{2}+\sum_{i=1}^{3}\gamma_{i}\|\upsilon_{i}(k)\|_{1}$, where $\pi_{7}\geq0$ and $\gamma_{i}\geq0$ for $i=1,2,3$, then we have that for $\delta\in(\xi,1)$, $\|\omega\|^{\delta,N+1}_{2}\leq\frac{\pi_{7}}{\delta}\|\omega\|_{2}^{\delta,N}+\sum_{i=1}^{3}\frac{\gamma_{i}}{\delta}\|\upsilon_{i}\|_{1}^{\delta,N}+\|\omega(0)\|_{2}$.
\end{lemma}
\begin{proof}
By multiplying both sides of the inequalities given in (1) and (2) by $\frac{1}{\delta^{k+1}}$ and performing a simple derivation, the proof can be completed.
\end{proof}

Upon utilizing Lemmas \ref{lemma 4} and \ref{lemma 5}, we can transform (\ref{proposition 2 eq}) given in Lemma \ref{proposition 2} into the following augmented error system:
\begin{align}\label{theorem 1 eq17}
\nonumber\renewcommand\arraystretch{1}\underbrace{\begin{bmatrix}
\|e_{\lambda}\|_{1}^{\delta,N+1}\\
\|e_{d}\|_{1}^{\delta,N+1}\\
\|e_{o}\|_{1}^{\delta,N+1}\\
\|e_{z}\|_{2}^{\delta,N+1}
\end{bmatrix}}\limits_{\bm{e}({N+1})}
\!\preceq\!\!\renewcommand\arraystretch{1}\underbrace{
\begin{bmatrix}
0&0&\frac{\pi_{3}}{\delta-\xi}&0\\
\frac{c_{4}}{\delta-\xi}&0&0&\frac{c_{4}}{\delta-\xi}\\
\frac{c_{4}M+c_{5}}{\delta-\xi}&\frac{c_{5}}{\delta}&0&\frac{c_{4}M+c_{5}}{\delta-\xi}\\
\frac{c_{7}}{\delta}&0&\frac{c_{6}}{\delta}&\frac{c_{14}}{\delta}
\end{bmatrix}}\limits_{T}\underbrace{
\begin{bmatrix}
\|e_{\lambda}\|_{1}^{\delta,N}\\
\|e_{d}\|_{1}^{\delta,N}\\
\|e_{o}\|_{1}^{\delta,N}\\
\|e_{z}\|_{2}^{\delta,N}
\end{bmatrix}}\limits_{\bm{e}({N})}\!+\bm{\theta},\\
\end{align}
where $\bm{\theta}=[
(\frac{\pi_{2}}{\delta}+1)\|e_{\lambda}(0)\|_{1},\
\frac{\varpi}{\delta}+\|e_{d}(0)\|_{1}, \
\frac{\varpi M}{\delta}+\|e_{o}(0)\|_{1},\
\|e_{z}(0)\|_{2}]^{\top}$, $\vartheta<L$, and $c_{14}=1-\beta\vartheta>1-\beta L=c_{8}$.

\textcolor[rgb]{0.15294,0.25098,0.5451}{\textbf{\emph{Part III}.}} \emph{Proof of $R$-linear convergence rate:}

Our main line is to prove that $\bm{e}(N)$ in (\ref{theorem 1 eq17}) vanishes linearly. To this end, we present two lemmas that serve as key ingredients in establishing the desired $R$-linear convergence rate.

\begin{lemma}\label{lemma 6}
Consider the augmented error system \eqref{theorem 1 eq17}. If the spectral radius of the matrix $T$ is less than 1, namely, $\rho(T)<1$, then $\|e_{\lambda}\|_{1}^{\delta,N}$, $\|e_{d}\|_{1}^{\delta,N}$, $\|e_{o}\|_{1}^{\delta,N}$ and $\|e_{z}\|_{2}^{\delta,N}$ are bounded. Moreover, the error quantities $\|e_{\lambda}(k)\|_{1}$, $\|e_{d}(k)\|_{1}$, $\|e_{o}(k)\|_{1}$ and $\|e_{z}(k)\|_{2}$ vanish at an $R$-linear rate of $\mathcal{O}(\delta^{k})$.
\end{lemma}

\begin{proof}
Letting $N\rightarrow\infty$ and iterating (\ref{theorem 1 eq17}) results in
$
\bm{e}(N+1)\preceq T^{2}\bm{e}(N-1)+T\bm{\theta}+\bm{\theta}
\preceq T^{N+1}\bm{e}(0)+\sum_{s=0}^{N}T^{s}\bm{\theta}.
$
From \cite[Ch. 5.6]{Horn1990}, we know that if $\rho(T)<1$, then
$\lim\nolimits_{s\rightarrow\infty}T^{s}=0 \ \text{and} \ \sum_{s=0}^{\infty}T^{s}=(I-T)^{-1}$.
And hence, we have
$
\lim\nolimits_{N\rightarrow\infty}\bm{e}(N+1)\preceq(I-T)^{-1}\bm{\theta}
$,
which, together with \eqref{theorem 1 eq17}, shows that $\|e_{\lambda}\|_{1}^{\delta,N+1}$, $\|e_{d}\|_{1}^{\delta,N+1}$, $\|e_{o}\|_{1}^{\delta,N+1}$, and $\|e_{z}\|_{2}^{\delta,N+1}$ are bounded. Upon utilizing Lemma \ref{lemma 4} again, we can conclude that $\|e_{\lambda}(k)\|_{1}$, $\|e_{d}(k)\|_{1}$, $\|e_{o}(k)\|_{1}$ and $\|e_{z}(k)\|_{2}$ converge at an $R$-linear rate of $\mathcal{O}(\delta^{k})$.
\end{proof}

The following lemma provides conditions to guarantee $\rho(T)<1$. This lemma initially appeared in \cite[Lemma 5]{Pu2021Nedic} and was introduced for a $3\times3$ matrix $T$. It can be readily extended to an $n \times n$ matrix $T$ by following a similar proof, and thus, we omit the details.

\begin{lemma}\label{lemma 7}
For a nonnegative and irreducible matrix 
$T\!=\![t_{ij}]_{i,j=1}^{n}$ with $0\leq
t_{ii}\!<\!1$, it holds that $\rho(T)<1$ if and only if $\det(I\!-\!T)\!>\!0$.
\end{lemma}

By Lemma \ref{lemma 7}, it is necessary to guarantee $\det(I\!-\!T)\!>\!0$ such that $\rho(T)\!<\!1$, i.e.,
$\det(I\!-\!T)
\!\!=\!\!\left(1\!\!-\!\!\frac{c_{14}}{\delta}\right)\left(1\!-\!\pi_{3}\left(\frac{c_{4}c_{5}}{\delta(\delta\!-\!\xi)^{2}}\!\!+\!\!\frac{c_{4}M\!+\!c_{5}}{(\delta\!-\!\xi)^{2}}\right)\!\right)\!-\!\left(\!\!c_{6}\!+\!\frac{\pi_{3}c_{7}}{\delta\!-\!\xi}\!\!\right)\!\!\left(\frac{c_{4}c_{5}}{\delta^{2}(\delta\!-\!\xi)}\!+\!\frac{c_{4}M\!+\!c_{5}}{\delta(\delta\!-\!\xi)}\!\right)\!>\!\!\!\left(1\!\!-\!\!\frac{c_{14}}{\delta}\!\right)\!\left(1\!-\!\pi_{3}\left(\!\frac{c_{4}c_{5}}{\delta(\delta\!-\!\xi)^{2}}\!+\!\frac{c_{4}M\!+\!c_{5}}{(\delta\!-\!\xi)^{2}}\!\right)\!\!\right)$\\$-\frac{c_{9}c_{10}\beta}{\delta}\left(\frac{c_{4}c_{5}}{\delta(\delta-\xi)^{2}}+\frac{c_{4}M+c_{5}}{(\delta-\xi)^{2}}\right)>0$,
where $c_{9}$ is a constant satisfying $0<c_{9}<\vartheta$ and $c_{10}=\left((2+\pi_{4})\pi_{1}M+\frac{\sqrt{M}}{\hat{M}}\right)(M\eta+1)\frac{1}{c_{9}}$. To ensure that $0\leq t_{ii}<1$ in (\ref{theorem 1 eq17}) and $\det(I-T)>0$, the following conditions are sufficient:
\vspace{-0.2cm}
\begin{subequations}
\begin{align}
1-\beta\vartheta&\geq 0,\label{theorem 1 eq19}\\
1-\frac{1-\beta\vartheta}{\delta}-\frac{c_{9}\beta}{\delta}&\geq0,\label{theorem 1 eq20}\\
1-\frac{(c_{11}+c_{10}L)(c_{12}+c_{13})\beta}{(\delta-\xi)^{2}}&\geq0,\label{theorem 1 eq21}
\end{align}
\end{subequations}
where $c_{11}=\pi_{1}M(M\eta+1)$, $c_{12}=c_{3}\pi_{4}^{2}(\xi^{-1}+\xi^{-2})\sqrt{N\rho}M$ and $c_{13}=M\pi_{4}(c_{3}(1+\xi^{-1})+\sqrt{N\rho})$.
In the light of the condition $\mathcal{L}=(c_{11}+c_{10}L)(c_{12}+c_{13})$ in Theorem \ref{theorem 1}, it is easy to check that $\mathcal{L}>L$. Recall that $\beta<\beta_{2}=\frac{(1-\xi)^{2}}{\mathcal{L}}$ in Theorem \ref{theorem 1} and $\vartheta<L$, it is easy to get $\beta<\frac{(1-\xi)^{2}}{\mathcal{L}}<\frac{1}{L}<\frac{1}{\vartheta}$, which means that \eqref{theorem 1 eq19} is fulfilled. To ensure \eqref{theorem 1 eq20} and \eqref{theorem 1 eq21}, we let
\vspace{-0.15cm}
\begin{subequations}\label{theorem 1 eq25}
\begin{align}
\delta&\geq 1-(\vartheta-c_{9})\beta,\label{theorem 1 eq26}\\
\delta&\geq \sqrt{\mathcal{L}\beta}+\xi, \label{theorem 1 eq27}
\end{align}
\end{subequations}
where $\mathcal{L}=(c_{11}+c_{10}L)(c_{12}+c_{13})$. To ensure \eqref{theorem 1 eq20} and \eqref{theorem 1 eq21}, combining \eqref{theorem 1 eq26} and \eqref{theorem 1 eq27}, we can take
\vspace{-0.1cm}
\begin{align}\label{theorem 1 eq28}
\delta=\max\left\{1-(\vartheta-c_{9})\beta,\sqrt{\mathcal{L}\beta}+\xi\right\}.
\end{align}
From $0<\beta<\beta_{2}=\frac{(1-\xi)^{2}}{\mathcal{L}}$ and \eqref{theorem 1 eq28}, we have $\delta\in(\xi,1)$. By $0<\beta<\frac{(1-\xi)^{2}}{\mathcal{L}}$ in Theorem \ref{theorem 1}, we observe that as $\beta$ increases from $0$ to $\beta_{2}$, the first argument inside the $\max$ operator in \eqref{theorem 1 eq28} decreases monotonically from $1$ to $1-(\vartheta-c_{9})\beta_{2}$, while the second argument increases monotonically from $\xi$ to $1$. By setting $1-(\vartheta-c_{9})\beta=\sqrt{\mathcal{L}\beta}+\xi$, we can obtain the solution
$\beta_{1}=\left(\frac{\sqrt{\mathcal{L}+4(\vartheta-c_{9})(1-\xi)}-\sqrt{\mathcal{L}}}{2(\vartheta-c_{9})}\right)^{2}$.
Therefore, we derive the desired result \eqref{theorem 1 eq3} in Theorem \ref{theorem 1}.
Combining \eqref{theorem 1 eq28} and \eqref{theorem 1 eq3}, we know that \eqref{theorem 1 eq26} and \eqref{theorem 1 eq27} hold, which implies $\rho(T)<1$. Upon using Lemma \ref{lemma 6}, we conclude that the error quantities $\|e_{\lambda}(k)\|_{1}$, $\|e_{d}(k)\|_{1}$, $\|e_{o}(k)\|_{1}$, and $\|e_{z}(k)\|_{2}$ decay at an $R$-linear rate of $\mathcal{O}(\delta^{k})$. Therefore, the desired result, $\|\bm{\lambda}(k+1)-\mathbf{1}\otimes \lambda_{\star}\|_{2}=\mathcal{O}(\delta^{k})$, follows readily from $\|e_{\lambda}(k)\|_{1}=\mathcal{O}(\delta^{k})$ and $\|e_{z}(k)\|_{2}=\mathcal{O}(\delta^{k})$.

By the optimality conditions \cite{Boyd2004}, \eqref{augment system4} can be equivalently expressed as
$\boldsymbol{u}^{i}(k+1)=\text{proj}_{\mathcal{U}^{i}_{T}(x^{i})}\left(\nabla \tilde{C}_{i}^{-1}(\tilde{\lambda}^{i}(k+1))\right)$, where $\nabla\tilde{C}_{i}^{-1}(\tilde{\lambda}^{i}(k+1))$ is the inverse function of $\nabla\tilde{C}_{i}(\boldsymbol{u}^{i})$,
$\tilde{\lambda}^{i}(k+1)=-(H^{i})^{\top}\lambda^{i}(k+1)$, and $\tilde{C}_{i}(\boldsymbol{u}^{i})=J^{i}(x^{i},\boldsymbol{u}^{i})$.
One then has that
\begin{align}
\nonumber&\left\|\boldsymbol{u}^{i}(k+1)-\boldsymbol{u}^{i}_{\star}\right\|_{2}\\
\nonumber&=\left\|\text{proj}_{\mathcal{U}^{i}_{T}(x^{i})}\left(\nabla\tilde{C}_{i}^{-1}(\tilde{\lambda}^{i}(k+1))\right)-\text{proj}_{\mathcal{U}^{i}_{T}(x^{i})}\left(\nabla\tilde{C}_{i}^{-1}(\tilde{\lambda}_{\star})\right)\right\|_{2}\\
\nonumber&\leq\left\|\nabla\tilde{C}_{i}^{-1}(\tilde{\lambda}^{i}(k+1))-\nabla\tilde{C}_{i}^{-1}(\tilde{\lambda}_{\star})\right\|_{2}\\
\nonumber&\leq\left\|\nabla^{2}\tilde{C}_{i}^{-1}((1\!-\!s)\tilde{\lambda}^{i}(k+1)+s\tilde{\lambda}_{\star})\right\|_{2}\left\|\tilde{\lambda}^{i}(k+1)-\tilde{\lambda}_{\star}\right\|_{2}\\
&\leq\frac{1}{\mu^{i}}\left\|(H^{i})^{\top}\right\|_{2}\left\|\lambda^{i}(k+1)-\lambda_{\star}\right\|_{2}=\mathcal{O}(\delta^{k}),
\end{align}
where $\tilde{\lambda}_{\star}=-(H^{i})^{\top}\lambda_{\star}$, the first inequality follows
from the standard non-expansiveness property of the projection operator \cite{Bertsekas1999}, the second inequality is obtained by utilizing the mean value
theorem and some constant $0<s<1$, the last inequality follows readily from $\nabla^{2}\tilde{C}_{i}(\boldsymbol{u}^{i})=\nabla^{2}_{\boldsymbol{u}^{i}}J^{i}(x^{i},\boldsymbol{u}^{i})\geq\mu^{i}I$, and the last equality follows from $\|\bm{\lambda}(k+1)-\mathbf{1}\otimes \lambda_{\star}\|_{2}=\mathcal{O}(\delta^{k})$.
The proof is thus completed.

\section*{Appendix III: Proof of Proposition \ref{lemma 8}}

We prove Proposition \ref{lemma 8} through the following two steps.

\emph{Step-1:} It needs to find one iteration $k_{i}$ such that \eqref{lemma 8 eq1} and \eqref{lemma 8 eq2} hold simultaneously. It follows from Theorem \ref{theorem 1} that $\lambda^{i}(k)\rightarrow\lambda_{\star}$ and $\boldsymbol{u}^{i}(k)\rightarrow\boldsymbol{u}^{i}_{\star}$ as $k\rightarrow\infty$, which implies that $\lambda^{i}_{k_{i}}\rightarrow\lambda_{\star}$ and $\boldsymbol{u}^{i}_{k_{i}}\rightarrow\boldsymbol{u}^{i}_{\star}$ as $k_{i}\rightarrow\infty$. One then has that $\lim_{k_{i}\rightarrow\infty}\mathfrak{g}^{i}(x^{i},\boldsymbol{u}^{i}_{k_{i}+1})=\lim_{k_{i}\rightarrow\infty}\mathfrak{g}^{i}(x^{i},\boldsymbol{u}^{i}_{k_{i}})$, which means that there exists at least one iteration $k_{i}$ such that \eqref{lemma 8 eq1} holds. We now prove \eqref{lemma 8 eq2}. Combining  $\lambda^{i}_{k_{i}}\rightarrow\lambda_{\star}$ and $\boldsymbol{u}^{i}_{k_{i}}\rightarrow\boldsymbol{u}^{i}_{\star}$ with $k_{i}\rightarrow\infty$, as well as the complementary slackness condition, one readily gives
$\lim_{k_{i}\rightarrow\infty}f^{i}(\lambda^{i}_{k_{i}+1})=\lim_{k_{i}\rightarrow\infty}f^{i}(\lambda^{i}_{k_{i}})=-J^{i}(x^{i},\boldsymbol{u}^{i}_{\star})$.  By combining this relation and $\lim_{k_{i}\rightarrow\infty}J^{i}(x^{i},\boldsymbol{u}^{i}_{k_{i}+1})=J^{i}(x^{i},\boldsymbol{u}^{i}_{\star})$, we have that
$\lim_{k_{i}\rightarrow\infty}(\|\nabla f^{i}(\lambda^{i}_{k_{i}})\|_{2}\|\lambda^{i}_{k_{i}+1}-\lambda^{i}_{k_{i}}\|_{2}+f^{i}(\lambda^{i}_{k_{i}+1})+J^{i}(x^{i},\boldsymbol{u}^{i}_{k_{i}+1}))=0$, which implies that there exists at least one iteration $k_{i}$ such that \eqref{lemma 8 eq2} is satisfied. Recall that Algorithm 1 converges linearly, and hence, we conclude that there exists at least one iteration $k_{i}$ such that \eqref{lemma 8 eq1} and \eqref{lemma 8 eq2} hold simultaneously.

\emph{Step-2:} We prove that \eqref{termination criteria eq1} and \eqref{termination criteria eq2} can also be readily satisfied when each subsystem satisfies \eqref{lemma 8 eq1} and \eqref{lemma 8 eq2}. From $\lambda^{i}_{k_{i}}\rightarrow\lambda_{\star}$ and $\boldsymbol{u}^{i}_{k_{i}}\rightarrow\boldsymbol{u}^{i}_{\star}$ with $k_{i}\rightarrow\infty$ for all $i$, we know that the primal feasibility is naturally satisfied, i.e., $\sum_{i=1}^{M}\mathfrak{g}^{i}(x^{i},\boldsymbol{u}^{i}_{k_{i}})\preceq b(\epsilon)$ as $k_{i}\rightarrow\infty$ (see the primal feasibility in KKT conditions).
Since $\epsilon-\epsilon_{b}$ is sufficiently small, we see that $\boldsymbol{u}^{i}_{k_{i}}$ goes to a sufficiently small neighborhood of the optimal solution $\boldsymbol{u}^{i}_{\star}$ when \eqref{lemma 8 eq1} holds. In this case, it is evident that $\sum_{i=1}^{M}\mathfrak{g}^{i}(x^{i},\boldsymbol{u}^{i}_{k_{i}})-b(\epsilon)\preceq\epsilon_{b}M\mathbf{1}_{N\rho}$. By combining this relation and summing both sides of \eqref{lemma 8 eq1}, it readily gives \eqref{termination criteria eq1} as expected.

We know that $-\sum_{i=1}^{M}f^{i}(\lambda^{i}_{k_{i}})$ is a lower bound of $\sum_{i=1}^{M}J^{i}(x^{i},\boldsymbol{u}^{i}_{\star})$, i.e., $-\sum_{i=1}^{M}f^{i}(\lambda^{i}_{k_{i}})\!\leq\!\sum_{i=1}^{M}J^{i}(x^{i},\boldsymbol{u}^{i}_{\star})$. Recalling that the function $f^{i}(\lambda^{i}_{k_{i}})$ is strongly convex w.r.t. $\lambda^{i}_{k_{i}}$, we have
$f^{i}(\lambda^{i}_{k_{i}})-f^{i}(\lambda^{i}_{k_{i}+1})\leq\|\nabla f^{i}(\lambda^{i}_{k_{i}})\|_{2}\|\lambda^{i}_{k_{i}}\!-\!\lambda^{i}_{k_{i}+1}\|_{2}
$. Then we have that $\sum\nolimits_{i=1}^{M}\|\nabla f^{i}(\lambda^{i}_{k_{i}})\|_{2}\|\lambda^{i}_{k_{i}+1}-\lambda^{i}_{k_{i}}\|_{2}+\sum\nolimits_{i=1}^{M}f^{i}(\lambda^{i}_{k_{i}+1})
\geq\sum\nolimits_{i=1}^{M}f^{i}(\lambda^{i}_{k_{i}})\geq-\sum\nolimits_{i=1}^{M}J^{i}(x^{i},\boldsymbol{u}^{i}_{\star})$. Summing both sides of \eqref{lemma 8 eq2}, it readily gives \eqref{termination criteria eq2} as expected. The proof is thus completed.

\section*{Appendix IV: Proof of Proposition \ref{lemma 9}}

Because $K^{i}$ is the optimal gain obtained from ARE, clearly, $K^{i}_{A}x^{i}(t)$ is the optimal solution to the optimization problem $\mathcal{P}(x)$ without constraints \eqref{constraint 0}-\eqref{tightening constraint}. Moreover, according to the definition of the terminal set $\mathcal{X}^{i}_{f}$ in \eqref{terminal set}, if all $x^{i}(t)\in\mathcal{X}^{i}_{f}$, $i\in \mathcal{V}$, it is easy to check that under the input $K^{i}_{A}x^{i}(t)$, the local constraint \eqref{constraint 0} and the global constraint \eqref{tightening constraint} in $\mathcal{P}(x)$ can be satisfied. And hence, we can conclude that for all $x^{i}(t)\in\mathcal{X}^{i}_{f}$, $i\in\mathcal{V}$, the optimization problem $\mathcal{P}(x)$ is equivalent to its unconstrained counterpart.

We now prove that if all $x^{i}(t)\in\mathcal{X}^{i}_{f}$, then the optimal solution can be obtained by executing one update using Algorithm 1, namely, $\bar{k}_{i}=1$. From the definition of $\mathcal{X}^{i}_{f}$ in \eqref{terminal set}, it is easy to check that $\mathfrak{g}^{i}(x^{i},\boldsymbol{u}_{i}^{1})\preceq(\frac{1}{M}-(N+1)\gamma)\mathbf{1}_{N\rho}\prec \frac{b(\epsilon)}{M}$ holds. Upon using the initial condition $\nabla f^{i}(\lambda^{i}_{0})=\mathbf{0}$, we know that $\mathfrak{g}^{i}(x^{i},\boldsymbol{u}^{i}_{0})=\frac{b(\epsilon)}{M}$. Therefore, the termination condition \eqref{lemma 8 eq1} is satisfied when $x^{i}(t)\in\mathcal{X}^{i}_{f}$ and $\bar{k}_{i}=1$.

Upon using the initial condition $z^{i}_{0}=\mathbf{0}$ for all subsystems, it is clear from \eqref{AA2} in Algorithm 1 that $\lambda^{i}_{1}=\mathbf{0}$. By the definition of $f^{i}(\lambda)$ (below \eqref{dual problem}), it readily gets $f^{i}(\lambda^{i}_{1})=-J^{i}(x^{i},\boldsymbol{u}^{i}_{1})$. Letting $k_{i}=0$ and $\bar{k}_{i}=k_{i}+1=1$, recalling $\nabla f^{i}(\lambda^{i}_{0})=\mathbf{0}$, substituting these results into \eqref{lemma 8 eq2} leads to $0\leq\frac{\epsilon_{g}}{M}$, which shows that the termination condition \eqref{lemma 8 eq2} holds when $\bar{k}_{i}=1$. The proof is thus completed.

\section*{Appendix V: Proof of Theorem 2}

We are ready to prove the recursive feasibility utilizing the standard induction. Recall that the problem $\mathcal{P}(x)$ is feasible at time $t=0$. Suppose that $\mathcal{P}(x)$ is feasible at time $t$. At time $t$, for each subsystem $i$, the control input vector is given by
\begin{align}\label{theorem 2 eq1}
\boldsymbol{\bar{u}}^{i}=\text{col}(u^{i}_{0},u^{i}_{1},\cdots,u^{i}_{N-1}).
\end{align}
Under the input vector \eqref{theorem 2 eq1}, the state vector obtained by evolving along the dynamics $x^{i}_{\ell+1}=A^{i}x^{i}_{\ell}+B^{i} u^{i}_{\ell}$ can be denoted by $\boldsymbol{\bar{x}}^{i}=\text{col}(x^{i}_{0},x^{i}_{1},\cdots,x^{i}_{N})$. We know that Algorithm 1 outputs $\boldsymbol{\bar{u}}^{i}$ when the termination criterion \eqref{lemma 8 eq} holds. Combined with Proposition \ref{lemma 8}, we obtain that $\sum_{i=1}^{M}\mathfrak{g}^{i}(x^{i},\boldsymbol{u}^{i}_{k_{i}+1})-b(\epsilon)\prec \epsilon M\mathbf{1}_{N\rho}$, which implies that
\begin{align}\label{therorem 2 eq2}
\sum_{i=1}^{M}\Big(\mathcal{C}^{i}x^{i}_{\ell}+\mathcal{D}^{i}u^{i}_{\ell}\Big)\prec (1-M\ell\epsilon)\mathbf{1}_{\rho}
\end{align}
holds for $\forall \ell\in\mathbb{Z}_{0}^{N-1}$. In what follows, it remains to prove that at the next time $t+1$, there is a feasible solution. To this end, at time $t+1$, a candidate input vector is selected as
\begin{align}\label{theorem 2 eq3}
\nonumber\boldsymbol{\bar{u}}^{i+}&=\text{col}(u^{i+}_{0},u^{i+}_{1},\cdots,u^{i+}_{N-1})\\
&=\text{col}(u^{i}_{1},\cdots,u^{i}_{N-1},K^{i}x^{i}_{N}).
\end{align}
At time $t+1$, consider the corresponding candidate state vector under (\ref{theorem 2 eq3})
\begin{align}\label{theorem 2 eq4}
\nonumber\boldsymbol{\bar{x}}^{i+}&=\text{col}(x^{i+}_{0},x^{i+}_{1},\cdots,x^{i+}_{N})\\
&=\text{col}(x^{i}_{1},\cdots,x^{i}_{N},A^{i}_{K}x^{i}_{N})
\end{align}
with $A^{i}_{K}=A^{i}+B^{i}K^{i}$. In the sequel, we prove the satisfaction of local and global constraints.

(1) \emph{Proof of the local constraint satisfaction:} Recall that at time $t$, $\mathcal{P}(x)$ is feasible. And hence, we get that  $x^{i}_{\ell}\in\mathcal{X}^{i}$ and $u^{i}_{\ell}\in\mathcal{U}^{i}$ holds for $\ell\in\mathbb{Z}_{0}^{N-1}$, and $x^{i}_{N}\in\mathcal{X}^{i}_{f}\subseteq\mathcal{X}^{i}$ holds. By the candidate input \eqref{theorem 2 eq3} and the candidate state \eqref{theorem 2 eq4}, we observe that $x^{i+}_{\ell}=x^{i}_{\ell+1}\in\mathcal{X}^{i}$ holds for  $i\in\mathbb{Z}_{0}^{N-1}$, and $u^{i+}_{\ell}=u^{i}_{\ell+1}\in\mathcal{U}^{i}$ holds for $\forall \ell\in\mathbb{Z}_{0}^{N-2}$. From $x^{i}_{N}\in\mathcal{X}^{i}_{f}$, it is easy to verify that $x^{i+}_{N}=A^{i}_{K}x^{i}_{N}\in\mathcal{X}^{i}_{f}$ and $u^{i+}_{N-1}=K^{i}x^{i}_{N}\in\mathcal{U}^{i}$ hold. Therefore, the local constraints are satisfied at time $t+1$.

(2) \emph{Proof of the global constraint satisfaction:} It readily follows from \eqref{theorem 2 eq3} and \eqref{theorem 2 eq4} that
$\sum_{i=1}^{M}\big(\mathcal{C}^{i}x^{i+}_{\ell}+\mathcal{D}^{i}u^{i+}_{\ell}\big)=\sum_{i=1}^{M}\big(\mathcal{C}^{i}x^{i}_{\ell+1}+\mathcal{D}^{i}u^{i}_{\ell+1}\big)\prec(1-M(\ell+1)\epsilon)\mathbf{1}_{\rho}
$ holds for $\ell\in\mathbb{Z}_{0}^{N-2}$. It is noted that at time $t$, $x^{i}_{N}\in\mathcal{X}^{i}_{f}$ holds for all $i\in\mathbb{Z}^{M}_{1}$. With the help of  $\sum_{i=1}^{M}\left(\mathcal{C}^{i}+\mathcal{D}^{i}K^{i}\right)x^{i}_{N}\preceq \sigma M\mathbf{1}_{\rho}$
for all $x^{i}_{N}\in\mathcal{X}^{i}_{f}$, as well as \eqref{theorem 2 eq3} and \eqref{theorem 2 eq4}, we obtain that for $\ell=N-1$, $\sum_{i=1}^{M}\big(\mathcal{C}^{i}x^{i+}_{N-1}+\mathcal{D}^{i}u^{i+}_{N-1}\big)=\sum_{i=1}^{M}\big(\mathcal{C}^{i}x^{i}_{N}+\mathcal{D}^{i}K^{i}x^{i}_{N}\big)\preceq(1-M(N+1)\gamma)\mathbf{1}_{\rho}\prec(1-MN\epsilon)\mathbf{1}_{\rho}$.
Therefore, one can conclude that for $\ell\in\mathbb{Z}_{0}^{N-1}$,
\begin{align}\label{theorem 2 eq5}
\sum_{i=1}^{M}\Big(\mathcal{C}^{i}x^{i+}_{\ell}+\mathcal{D}^{i}u^{i+}_{\ell}\Big)\prec(1-M(\ell+1)\epsilon)\mathbf{1}_{\rho},
\end{align}
which implies the satisfaction of the global constraint, that is, $\sum_{i=1}^{M}\mathfrak{g}^{i}(x^{i+},\boldsymbol{\bar{u}}^{i+})\prec b(\epsilon)$, where $\boldsymbol{\bar{u}}^{i+}=\text{col}(u^{i+}_{0},u^{i+}_{1},\cdots,u^{i+}_{N-1})$.

By combining (1) and (2), we conclude that $\boldsymbol{\bar{u}}^{i+}$ is a feasible solution, and the recursive feasibility is thus proved.

\end{document}